\documentclass[intlimits,12pt,reqno]{amsart}
\usepackage{latexsym,amsthm,amsmath,amssymb,mathrsfs,mathtools}
\usepackage[T1]{fontenc}
\usepackage[utf8]{inputenc}
\usepackage[mathscr]{eucal}
\usepackage{enumerate}
\usepackage{enumitem}
\usepackage{color}

\usepackage{ esint }
\usepackage[colorlinks	=	true,						
linkcolor	=	red,
urlcolor	=	red,
citecolor	=	red]%
{hyperref}
\usepackage{cleveref}
\usepackage{graphicx}
\usepackage{tikz}
\usetikzlibrary{patterns}

\newtheorem{theorem}{Theorem}[section]
\newtheorem{lemma}[theorem]{Lemma}

\newtheorem{proposition}[theorem]{Proposition}
\theoremstyle{definition}
\newtheorem{remark}[theorem]{Remark}
\newtheorem{definition}[theorem]{Definition}

\newtheorem*{theoremx}{Theorem}
\newtheorem*{propositionx}{Proposition}

\def\diam{\operatorname{diam}}
\def\card{\operatorname{card}}
\def\dist{\operatorname{dist}}

\DeclareMathOperator{\Mod}{mod}

\title{Uniformization of cofat domains on  metric two-spheres}
\author{Chengxi Li and Kai Rajala} 
\address{C.\ Li: Department of Mathematics and Statistics, University of Jyv\"askyl\"a, P.O. Box 35 (MaD), FI-40014, University of Jyv\"askyl\"a, Finland. {\tt cheli@jyu.fi }} %
\address{K.\ Rajala: Department of Mathematics and Statistics, University of Jyv\"askyl\"a, P.O. Box 35 (MaD), FI-40014, University of Jyv\"askyl\"a, Finland. {\tt kai.i.rajala@jyu.fi}}
\date{}

\begin{document}

	\thanks{  
		\newline C.L. was supported by the Research Council of Finland, project numbers 364210 and 360505, and China Scholarship Council Fellowship, project number 202306200022. K.R. was supported by the Research Council of Finland, project number 360505. }

	\begin{abstract}
		
		We extend \emph{Schramm's cofat uniformization theorem} to cofat domains on upper Ahlfors 2-regular metric two-spheres $X$. Specifically, we show that if $\Omega \subset X$ is a cofat domain, then there exists a $\frac{\pi}{2}$-quasiconformal homeomorphism $f: \Omega \to D$ onto a circle domain $D \subset \mathbb{S}^2$. Moreover, $f$ preserves the point-components and non-trivial complementary components. We also construct examples which show that the above conclusions are not true for countably connected $\ell^{\alpha}$-subdomains of $\mathbb{S}^2$.
		
	\end{abstract}
	
	\maketitle
	
	\section{Introduction}

	The Riemann mapping theorem asserts that every simply connected proper subdomain of the complex plane is conformally equivalent to the unit disk $\mathbb{D}$. For multiply connected domains, Koebe \cite{koebe1907uniformisierung} conjectured  that every domain in the Riemann sphere $\mathbb{S}^2$ is conformally equivalent to a \emph{circle domain}, i.e., a domain whose boundary components are points and circles.
	
	 Koebe himself proved the conjecture for finitely connected domains \cite{koebe1920abhandlungen,koebe1922konforme}, while He and Schramm \cite{he1993fixed} established it for countably connected domains. Subsequently, in \cite{schramm1995transboundary}, Schramm introduced a powerful tool for studying the uniformization of multiply connected domains—the transboundary extremal length (or \emph{the transboundary modulus}). He also proved \emph{the cofat uniformization theorem}, which proves Koebe's conjecture under geometric conditions; see below. Recently, Karafyllia and Ntalampekos \cite{karafyllia} gave an affirmative answer to Koebe's conjecture for the class of Gromov hyperbolic domains. The general case for uncountably connected domains remains an open problem. See also \cite{luo2024koebe,rajala2021uniformizations,younsi2016removability} for recent works on Koebe's conjecture.
	 
	 In this paper we extend \emph{Schramm’s cofat uniformization theorem} to cofat domains on upper Ahlfors 2-regular metric two-spheres.

	 Recall that a metric space $X=(X,d)$ is a \emph{metric two-sphere} if $X$ is homeomorphic to $\mathbb{S}^2$ and has finite two-dimensional Hausdorff measure $\mathcal{H}^2$. Given a domain $\Omega \subset X$, we call a connected component $p$ of $X \setminus \Omega$ \emph{non-trivial} and denote $p \in \mathcal{C}_N(\Omega)$ if $\diam(p)>0$. Otherwise we call $p$ a \emph{point-component} and denote $p \in \mathcal{C}_P(\Omega)$. 
	 
	 Let $\hat{\Omega}:=X / \sim$, where $z \sim w$ if either $ z=w \in \Omega$ or $z,w \in p$ for some $p \in \mathcal{C}(\Omega):=\mathcal{C}_N(\Omega) \cup \mathcal{C}_P(\Omega)$. We denote the induced quotient map by $\pi_\Omega: X \to \hat{\Omega}$. Every homeomorphism $f: \Omega \to D$ between subdomains of metric two-spheres can be extended to a homeomorphism $\hat{f}: \hat{\Omega} \to \hat{D}$. We make no distinction between $p \in \mathcal{C}(\Omega)$ and $\pi_\Omega(p) \in \hat{\Omega}$.
	 
	 Let $X$ be a metric two-sphere. We call $A\subset X$ \emph{$\tau$-fat} if for every $x_0 \in A$ and every open ball $B(x_0,r) \subset X$ that does not contain $A$, we have $$\mathcal{H}^2 (A \cap B(x_0,r)) \geq \tau r^2.$$ A domain $\Omega \subset X$ is \emph{$\tau$-cofat} if there exists $\tau >0$ such that every $p \in \mathcal{C}_N(\Omega)$ is $\tau$-fat. A domain $ D \subset \mathbb{S}^2 $ is called \emph{a circle domain} if each $p \in \mathcal{C}_N(D)$ is a disk. We now state our main result.
	 
	 \begin{theorem}\label{main}
	 	Suppose that $X$ is an upper Ahlfors $2$-regular metric two-sphere. If \,$\Omega \subset X$ is a cofat domain, then there exists a $\frac{\pi}{2}$-quasiconformal homeomorphism  $f\colon\Omega \to D$ onto a circle domain $D \subset \mathbb{S}^2$. Moreover, 
	 	\begin{equation} \label{claimi2}
	 		\hat{f}(\mathcal{C}_N(\Omega))=\mathcal{C}_N(D) \quad \text{and} \quad  \hat{f}(\mathcal{C}_P(\Omega))=\mathcal{C}_P(D). 
	 	\end{equation}
	 \end{theorem}
	 The upper $2$-regularity and quasiconformality are defined in \eqref{ahlfors} and Definition \ref{gd}, respectively. Schramm \cite{schramm1995transboundary} proved Theorem \ref{main} for $X=\mathbb{S}^2$ and a conformal $f$. Compactifications of $\mathbb{R}^2$ equipped with the $\ell^{\infty}$-norm show that the constant $\frac{\pi}{2}$ is best possible (see \cite{romney2019quasiconformal}). 

     The main difficulty in proving Theorem \ref{main} is the adaptation of basic transboundary modulus estimates to the metric space setting. Note that if quasiconformal maps preserved cofat domains, then Theorem \ref{main} would follow by combining Schramm's theorem with \cite[Theorem 1.6]{rajala2017uniformization}, according to which there is a quasiconformal map $\phi:X \to \mathbb{S}^2$. 

    Such an argument can be applied in \emph{linearly locally connected} (LLC) spaces $X$ (see e.g. \cite[Theorem 8.23]{heinonen2001lectures}), on which $\phi$ is \emph{quasisymmetric} by \cite[Theorem 1.1]{bonk2002quasisymmetric}, and therefore preserves cofat domains. On the other hand, we demonstrate in Section \ref{sec:cofatness-not-qc-invariant} that cofatness is not always preserved by quasiconformal maps in the setting of Theorem \ref{main} without the LLC-condition. 
	 
	 We now give some background. Non-smooth versions of the uniformization theorem have found applications in various areas of mathematics, such as the calculus of variations, metric geometry, complex dynamics, and 
	 geometric group theory (see e.g. \cite{bonk2006quasiconformal,bonk2011uniformization,bonk2005conformal,bonk2013quasisymmetric,haissinsky2009coarse,kleiner2006asymptotic}). A fundamental result by Bonk and Kleiner \cite{bonk2002quasisymmetric} shows that an Ahlfors 2-regular metric two-sphere admits a \emph{quasisymmetric} homeomorphism onto the standard $\mathbb{S}^2$ if and only if it is \emph{linearly locally connected}. 

 A related line of research concerns quasisymmetric uniformization of carpet-like domains. Bonk \cite{bonk2011uniformization} proved that a Sierpiński carpet in $\mathbb{S}^2$ whose peripheral circles are uniformly relatively separated uniform quasicircles is quasisymmetrically equivalent to a round carpet. Bonk and Merenkov \cite{bonk2013quasisymmetric} subsequently developed rigidity results for standard Sierpiński carpets, using discrete modulus as a quasisymmetric invariant. Ntalampekos \cite{ntalampekos2019non,Nta20,ntalampekos2021non,ntalampekos2025quasiconformal,ntalampekos2020non} further studied removability and uniformization problems for carpets.
 
 The series of works by Lytchak and Wenger \cite{lytchak2016regularity,lytchak2017area,lytchak2017energy,lytchak2017intrinsic,lytchak2018isoperimetric,lytchak2020canonical,lytchak2020dehn} establishes existence and regularity results for area-minimizing discs in proper metric spaces, and develops the intrinsic geometry of such minimizers, including their quasiconformal parametrizations under mild assumptions. 

The multiply connected case was considered by Merenkov and Wild\-rick \cite{merenkov2013quasisymmetric}, who proved a quasisymmetric version of Koebe's theorem for finitely and countably connected metric surfacs. See also \cite{cheeger2023thin,geyer2018quantitative,HaLi23,heinonen2010quasisymmetric,ikonen2022uniformization,karafyllia,ntalampekos2025uniformization,ntalampekos2026uniformization,pankka2014geometry,rajala2021uniformization,rajala2021quasisymmetric,rehmert2022quasisymmetric,wildrick2010quasisymmetric} for progress on non-smooth uniformization theory.

	Quasisymmetric maps from $X$ to $\mathbb{S}^2$ have globally controlled distortion, and can only exist on spaces whose geometries are controlled. Towards a more general uniformization theory, the second author \cite{rajala2017uniformization} established conditions for a metric surface to admit a \emph{quasiconformal} uniformization. Subsequently,  Ntalampekos and Romney \cite{ntalampekos2023polyhedral,ntalampekos2024polyhedral} and Meier and Wenger \cite{Mei,meier2024quasiconformal} (for a locally geodesic $X$) found \emph{weakly quasiconformal parametrizations} for \emph{all metric two-spheres}. 
	 
	 The purpose of our work is to initiate the study of quasiconformal uniformization for multiply connected subdomains of metric surfaces. It is desirable to find further conditions which are sufficient for the conclusions of Theorem \ref{main}. 
	
	One approach is to consider summability conditions on the diameters of the complementary components. As a byproduct of the proof of Theorem \ref{main}, we obtain the following result which is of independent interest even for $X=\mathbb{S}^2$.
	
	Given $\alpha>0$, let $\mathcal{F}_\alpha(X)$ be the collection of \emph{$\ell^\alpha$-domains} on $X$, i.e., the \emph{countably connected} subdomains $\Omega_\alpha$ of $X$ satisfying  
	$$\sum_{p \in \mathcal{C}(\Omega_\alpha)} \operatorname{diam}(p)^\alpha<\infty. $$ 
	
	 \begin{theorem}\label{dimi2}
		Suppose that $X$ is a metric two-sphere. If \,$\Omega_2 \in \mathcal{F}_2(X)$, then every conformal homeomorphism $f: \Omega_2 \to D$ onto a circle domain $D \subset \mathbb{S}^2$ satisfies $\hat{f}(\mathcal{C}_N(\Omega_2)) \subset \mathcal{C}_N(D)$.
	\end{theorem}
	See Remark \ref{diam2yx} for how to extract the proof of Theorem \ref{dimi2} from the proof of Theorem \ref{main}. Theorem \ref{dimi2} is no longer true if the countable connectedness assumption is removed from the definition of $\mathcal{F}_2(X)$, even when $X=\mathbb{S}^2$, see \cite[Theorem 4.1]{GehMar85}.  
	
	Our next result shows that the exponent $\alpha=2$ in Theorem \ref{dimi2} is sharp even when $X=\mathbb{S}^2$. 
	
	\begin{theorem}\label{nontop2}
		For every $\alpha>2$ there is a domain $U_\alpha \in \mathcal{F}_\alpha(\mathbb{S}^2)$, such that the closed unit disk $\overline{\mathbb{D}} \in \mathcal{C}_N(U_\alpha)$ but $\hat{f}(\overline{\mathbb{D}} ) \in \mathcal{C}_P(D)$ for every conformal homeomorphism $f: U_\alpha \to D$ onto a circle domain $D \subset \mathbb{S}^2$.
	\end{theorem}

	Finally, we demonstrate the failure of the second half of Condition \eqref{claimi2} in any $\mathcal{F}_\alpha(\mathbb{S}^2)$.
	\begin{theorem}\label{ptonon2}
		For every $\alpha>0$ there is a domain $\Omega_\alpha \in \mathcal{F}_\alpha(\mathbb{S}^2)$ such that $\{0\} \in \mathcal{C}_P(\Omega_\alpha)$ but $\hat{f}(\{0\}) \in \mathcal{C}_N(D)$ for every conformal homeomorphism $f: \Omega_\alpha \to D$ onto a circle domain $D \subset \mathbb{S}^2$.
	\end{theorem}
	Theorem \ref{ptonon2} generalizes a construction of Ntalampekos \cite[Section 6]{ntalampekos2023rigidity}, which covers the case $\alpha=2$.
	
	\subsection*{Acknowledgment}
We thank the referee for helpful comments. The first named author wishes to express his gratitude to his co-advisor, Professor Pekka Koskela, for his guidance and encouragement.



	\section{Preliminaries}\label{ybzs}
	
We denote the open ball in a metric space $X$ with center $x_0 \in X$ and radius $r>0$ by $B(x_0,r)$, and $S(x_0,r):=\{x \in X: d(x_0,x)=r\}$. The metric on $X$ is denoted by $d$, and we equip the standard $\mathbb{S}^2$ with the spherical metric $d_{\mathbb{S}^2}$. We call $X$ \emph{Ahlfors 2-regular} if there is an $\alpha \geq1$ such that
	\begin{equation}\label{ahlfors}
		\alpha^{-1} R^2 \leq \mathcal{H}^2(B(x, R)) \leq \alpha R^2
	\end{equation}
	for all $x \in X$ and $0 < R< \diam(X)$, and \emph{upper Ahlfors 2-regular} if the second inequality holds in \eqref{ahlfors}.

	
	\subsection{Conformal Modulus}
	 We recall the conformal modulus of curve families and give the geometric definition of quasiconformal maps.
	
	\begin{definition}[Conformal modulus]
		Let $\Gamma \subset X$ be a family of curves in a metric two-sphere $X$. The conformal modulus of $\Gamma$ is
		$$\Mod_\mathrm{C}(\Gamma)=\inf_{\rho \in Adm_\mathrm{C}(\Gamma)} \int_X \rho^2 \, d\mathcal{H}^2,$$
		where $Adm_\mathrm{C}(\Gamma)$ is the collection of admissible functions for $\Gamma$, i.e., all Borel functions $\rho: X \to [0,\infty]$ satisfying 
		$$\int_{\gamma} \rho \, ds \geq 1 \quad \quad \text{for every locally rectifiable curve }\gamma \in \Gamma. $$
		\end{definition}
	
	We now give the \emph{geometric definition of quasiconformality}. We assume that $X$ and $Y$ are metric two-spheres. 
	
	\begin{definition}\label{gd}
		A homeomorphism $f: \Omega \to D$ between domains $\Omega \subset X$ and $D \subset Y$ is \emph{$K$-quasiconformal}, or $K$-QC, if
		\begin{eqnarray*}
			K^{-1} \Mod_\mathrm{C}(\Gamma) \leq \Mod_\mathrm{C}(f \Gamma) \leq K \Mod_\mathrm{C}(\Gamma)
		\end{eqnarray*}
		for every curve family $\Gamma$ in $\Omega$. Here $f\Gamma :=\{f \circ \gamma: \gamma \in \Gamma \} $.
	\end{definition}
	
    We can also define quasiconformal maps between metric spaces via the so-called metric definition (see \cite[Section 14]{heinonensobolev}). In the setting of this paper, the geometric definition is more natural due to the uniformization theory which has recently been developed using modulus.

    We will apply the following uniformization result by the second author \cite[Theorem 1.4-1.6]{rajala2017uniformization} to establish the existence of quasiconformal maps in Theorem \ref{main}. The sharp constant $\frac{\pi}{2}$ was proved by Romney \cite{romney2019quasiconformal}.
    
    \begin{theorem}\label{rajala}
    	Suppose that $X$ is an upper Ahlfors 2-regular metric two-sphere. Then there exists a $\frac{\pi}{2}$-QC map $h:X \to\mathbb{S}^2$.
    \end{theorem}

    We also need the stability of quasiconformal maps under locally uniform convergence. See \cite[Theorem 7.1]{koskela2009lectures} for a proof of the following proposition.
    \begin{proposition}\label{stable}
    	Let $f_j: \Omega \to D_j\subset \mathbb{S}^2$ be $K$-QC for each $j \geq 1$, and suppose that $f_j $ converges to $ f: \Omega \to D\subset \mathbb{S}^2 $ locally uniformly. If $f$ is a homeomorphism, then $f$ is $K$-QC.
    \end{proposition}

	\subsection{Transboundary Modulus}
	We will apply Schramm's transboundary modulus \cite{schramm1995transboundary}. The decomposition space $\hat{\Omega}$ associated to a domain $\Omega \subset X$ is defined in the introduction before Theorem \ref{main}.

	\begin{definition}
		Let $X$ be a metric two-sphere and fix a domain $\Omega \subset X$. The \emph{transboundary modulus} $\Mod_\mathrm{T}(\Gamma)$ of a family $\Gamma$ of curves in $\hat{\Omega}$ is 
		$$
		\Mod_\mathrm{T}(\Gamma)=\inf_{\rho \in Adm_\mathrm{T}(\Gamma)} \int_{\Omega} \rho^2 \, d\mathcal{H}^2 + \sum_{p \in \mathcal{C}(\Omega)} \rho(p)^2, 
		$$
		where $Adm_\mathrm{T}(\Gamma)$ is the collection of \emph{admissible functions for $\Gamma$}, i.e., Borel measurable functions $\rho\colon \hat{\Omega}\to [0,\infty]$ for which 
		$$
		1 \leq \int_{\gamma} \rho \, ds +\sum_{p \in \mathcal{C}(\Omega) \cap |\gamma|} \rho(p)  \quad \text{for all } \gamma \in \Gamma.  
		$$ 
	\end{definition}
	Here $|\gamma|$ denotes the trajectory of the curve $\gamma$ and $\int_{\gamma} \rho \, ds$ is the curve integral of the restriction of $\gamma$ to $\Omega$. More precisely, the restriction is 
	a countable union of disjoint curves $\gamma_j$, each of which maps onto a component of $|\gamma| \setminus \mathcal{C}(\Omega)$, and we can define 
	$$
	\int_{\gamma} \rho \, ds = \sum_j \int_{\gamma_j} \rho \, ds.  
	$$
	
	We will apply the quasi-invariance of the transboundary modulus under quasiconformal maps between metric surfaces. 
	\begin{theorem}\label{qcinvariance}
		Let $\Omega$ and $D$ be subdomains of metric surfaces $X$ and $Y$, respectively. If $f: \Omega \to D$ is a $K$-QC map, then for every curve family $\Gamma$ in $\hat{\Omega}$, we have 
		\begin{eqnarray*}
			K^{-1} \Mod_\mathrm{T}(\Gamma) \leq \Mod_\mathrm{T}(f \Gamma) \leq K \Mod_\mathrm{T}(\Gamma).
		\end{eqnarray*}
	\end{theorem}
	
	The proof of Theorem \ref{qcinvariance} involves notions from the Sobolev theory in metric measure spaces, such as upper gradients. Since we will not apply such notions elsewhere in this paper, we do not give the definitions here. Instead, we refer to \cite{heinonensobolev,williams2012geometric} for the definitions and basic properties. 
	\begin{proof}
		Let $\rho \in Adm_\mathrm{T}(\Gamma)$. By Definition \ref{gd}, $f^{-1}: D\to \Omega$ is also $K$-QC. It then follows from Williams' theorem \cite[Theorem 1.1]{williams2012geometric} that $f^{-1}$ belongs to the Newtonian 
		Sobolev space $N_{loc}^{1,2}(D, \Omega)$. Moreover, for $\mathcal{H}^2$-almost every $ z\in D$, we have
		\begin{equation}\label{upp}
			g_{f^{-1}}(z)^2 \leq K \cdot J_{f^{-1}}(z).
		\end{equation}
		Here $g_{f^{-1}}$ and $J_{f^{-1}}$ denote the minimal weak upper gradient and the volume
		derivative of $f^{-1}$, respectively. 
		
		Then we can define $\tilde{\rho}(z)=\rho(f^{-1}(z)) \cdot g_{f^{-1}}(z)$ if $z \in D$ and $\tilde{\rho}(p)=\rho(\hat{f}^{-1}(p))$ if $p\in \mathcal{C}(D)$. We claim that $\tilde{\rho}: \hat{D}\to[0,\infty]$ is admissible for $\hat{f}\Gamma$. Indeed, for almost every $\widetilde{\gamma}:=\hat{f} \circ \gamma \in \hat{f}\Gamma$, we have
		$$\begin{aligned} \int_{\widetilde{\gamma}} \tilde{\rho} \,d s+\sum_{p\in \mathcal{C}(D) \cap|\widetilde{\gamma}|} \tilde{\rho}(p) & =\int_{\widetilde{\gamma}} \rho\left(f^{-1}(z)\right)\cdot g_{f^{-1}}(z)\,d s+\sum_{\bar{p} \in \mathcal{C}(\Omega) \cap |\gamma|} \rho(\bar{p}) \\ & \geq \int_\gamma \rho(u) \,d s+\sum_{\bar{p} \in \mathcal{C}(\Omega) \cap |\gamma|}  \rho(\bar{p}) \\ & \geq 1.\end{aligned}$$
		
		Here the first inequality holds since $g_{f^{-1}}$ is a weak upper gradient of $f^{-1}$. Combining with \eqref{upp}, we can estimate $\Mod_\mathrm{T}(\hat{f}\Gamma)$:
		$$\Mod_\mathrm{T} (\hat{f}\Gamma) \leq \int_D (\tilde{\rho})^2 \,d \mathcal{H}^2+\sum_{p \in \mathcal{C}(D)} \tilde{\rho}(p)^2 \leq K\Big[\int_{\Omega} \rho^2 \,d \mathcal{H}^2+\sum_{p \in \mathcal{C}(\Omega)}{\rho}(p)^2 \Big]$$
		for all $\rho\in Adm_\mathrm{T}(\Gamma)$, so we have $\Mod_\mathrm{T}(\hat{f} \Gamma) \leq K \Mod_\mathrm{T}(\Gamma)$. The reverse inequality $\Mod_\mathrm{T}(\Gamma) \leq K \Mod_\mathrm{T}(\hat{f}\Gamma)$ is proved in a similar way. \end{proof}


 \section{Approximation with finitely connected domains}\label{existence}
    In this section we start the proof of Theorem \ref{main}. We use the following approximation scheme.
    
   Fix a cofat domain $\Omega \subset X$ on an upper Ahlfors 2-regular metric two-sphere. By Theorem \ref{rajala} there exists a $\frac{\pi}{2}$-QC homeomorphism $h: \Omega \to \Omega'$ onto a domain $\Omega':= h(\Omega) \subset \mathbb{S}^2$.
    	
    	 In order to prove Theorem \ref{main}, we may assume that $\card \mathcal{C}_N(\Omega)=\infty$, since otherwise the existence of a conformal homeomorphism $g: \Omega'\to D$ onto a circle domain $D \subset \mathbb{S}^2$ follows from Koebe's theorem \cite[Theorem 9.5]{bonk2011uniformization}, so we get a $\frac{\pi}{2}$-QC map $f:=g \circ h: \Omega \to D$; notice that Definition \ref{gd} yields that the composition of a $K$-QC map and a conformal map is $K$-QC.  
    	 
    	 Since $\Omega$ is cofat, each $p \in C_N(\Omega)$ has positive Hausdorff $2$-measure. In particular, the collection $C_N(\Omega)$ is countable, and we can enumerate $C_N(\Omega)=\left\{p_0, p_1, p_2, \ldots\right\}$. Let $\Omega_{j} := X \setminus \left\{p_0, p_1, \ldots, p_j\right\} $ be a sequence of finitely connected domains, so that $\Omega_1 \supset \Omega_2 \supset \cdots \supset \Omega$. 
	 
	 In order to prove Theorem \ref{main},  we may assume that  
	 $$
    	\widetilde{\Omega}:= X \setminus \overline{\bigcup_{p \in \mathcal{C}_N(\Omega)}  p}=\Omega. 
    	$$
    	Indeed, if we can prove the existence of a homeomorphism $f:\tilde \Omega \to D$ as in Theorem \ref{main}, then also the restriction of $f$ to $\Omega$ satisfies the conditions in Theorem \ref{main}. 
    	
	By Theorem \ref{rajala} we know that there are $\frac{\pi}{2}$-QC maps $h_j: \Omega_{j} \to \Omega_j'$, such that each $ \Omega_j' \subset \mathbb{S}^2$ is a finitely connected domain on $\mathbb{S}^2$.
    	Now by Koebe's theorem, we can find conformal homeomorphisms $g_j:\Omega_j' \rightarrow D_j \subset \mathbb{S}^2$ onto circle domains $D_j \subset \mathbb{S}^2$. Moreover, 
	$$q_{j,\ell}:=\hat{g}_j(\hat{h}_j(p_\ell))$$ is a disk for all $\ell=0,1,\ldots,j$; here $ \hat{h}_j: \hat{\Omega}_j \to \hat{\Omega}_j'$ and $\hat{g}_j: \hat{\Omega}_j' \to \hat{D}_j$ are the unique homeomorphic extensions of $h_j$ and $g_j$.
    	By postcomposing with Möbius transformations on $\mathbb{S}^2$, we may assume that 
    	\begin{equation} \label{normal}
    		q_{j,0}= \mathbb{S}^2 \setminus \mathbb{D} \quad \text{for all } j=1,2,\ldots. 
    	\end{equation}
    	
    	We get a sequence of quasiconformal maps $(f_j)_j$, where
    	$f_j := g_j \circ h_j: \Omega_j \rightarrow D_j$. Each $f_j$ maps $\Omega_j \subset X$ onto a circle domain $D_j \subset \mathbb{S}^2$.

    	For every $\ell=0,1,2,\ldots$, any subsequence of $(q_{j,\ell})_j$ has a further subsequence Hausdorff converging to a limit disk or a point. Therefore we can use a diagonal argument to find a subsequence $(f_{j_k})_k$, converging locally uniformly in $\Omega$, so that $q_{j_k,\ell}\to q_\ell$ for each $\ell$. We continue to denote the subsequence by $ (f_{j})_j$. By Condition \eqref{normal} and Proposition \ref{stable}, the limit map $f$ is non-constant and therefore a $\frac{\pi}{2}$-QC map from $\Omega$ onto a domain $D:= f(\Omega)\subset \mathbb{S}^2$. Each $q_\ell$, $\ell=0,1,2,\ldots$, is a disk or a point, and $q_0=\mathbb{S}^2 \setminus \mathbb{D}$. 

       Theorem \ref{main} now follows once we establish the following. 
   
       	\begin{theorem}\label{bz}
       	The $\frac{\pi}{2}$-QC map $f: \Omega \to D$ has the following properties:
       	\begin{eqnarray}
       		\label{pro1}
       		& & \diam(\hat{f}(p))=0  \quad \text{for all } p \in \mathcal{C}_P(\Omega), \\
       		\label{pro2}
       		& & q_\ell=\hat{f}(p_\ell) \quad \text{and}\quad \diam(q_\ell)>0  \quad  \text{for all } \ell=0,1,2,\ldots. 
       	\end{eqnarray} 
       	Here $\hat{f}: \hat{\Omega} \to \hat{D}$ denotes the homeomorphism induced by $f$.
       \end{theorem}


	\section{Modulus estimates}\label{est}
	The aim of the following is to develop the estimates required for the proof of Theorem \ref{bz}. 
	
	In this section we assume that $X$ is a metric two-sphere that satisfies the upper Ahlfors $2$-regularity condition \eqref{ahlfors} with constant $\alpha \geq 1$. Recall that if $\Omega \subset X$ is a domain, then we denote the quotient map by $\pi_{\Omega}: X \to \hat{\Omega}$. We also denote by $A(x_0,r) \subset X$ the annulus $ B(x_0, 2r) \setminus\overline{B}(x_0, r)$. 
	
	Given a domain $\Omega \subset X$,  $\bar p \in \hat \Omega$, $x_0 \in \bar p$, and $r>0$, we denote the family of curves joining $\pi_{\Omega}(S(x_0,r))$ and $\pi_{\Omega}(S(x_0,2r))$ in 
	$\pi_{\Omega}({\overline{A}(x_0,r)})$ by $\Gamma_{x_0, r}$. Recall that $\bar p$ may or may not be a point of $\Omega$, and that if it is not then 
	we apply the notation for $\bar p$ also for $\pi^{-1}_{\Omega}(\bar p) \in \mathcal{C}(\Omega)$. The following estimate is the main technical result towards Theorem \ref{main}. 
	
	\begin{proposition} \label{unibdd}
		Let $\Omega \subset X$ be a finitely connected $\tau$-cofat domain. There is an $M>0$ depending only on $\alpha$ and $\tau$ so that $\Mod_\mathrm{T} \Gamma_{x_0, r} \leq M$.
	\end{proposition} 
	
	We postpone the proof of Proposition \ref{unibdd} until Section \ref{proofprop}. We now fix a $\tau$-cofat domain as in Theorem \ref{main}. Moreover, given $j=1,2,\ldots$, we let $\Omega_{j}$ be the finitely connected domain constructed in Section \ref{existence}. We fix $\bar p \in \hat{\Omega}$ and $x_0 \in \bar p$ as above, and choose radii $R,r>0$ such that $100r <R<\frac{\diam(X)}{2}$. We denote by $\Gamma_{R, r}^j$ the family of curves joining $\pi_{\Omega_j}(S(x_0,r))$ and $\pi_{\Omega_j}(S(x_0,R))$ in $\hat{\Omega}_j \setminus \{\bar p\}$.
	
		\begin{proposition}\label{asyb}
	The following estimate holds:
		\begin{eqnarray*}
			\limsup _{j \rightarrow \infty} \Mod_\mathrm{T} \Gamma_{R, r}^j \leq \lambda(r), \quad \text{ where } \lambda(r) \to 0 \text{ as } \; r \to 0.
		\end{eqnarray*}
	\end{proposition}
	 \begin{proof} 
		We choose a sequence of radii $R_n$ decreasing to zero as follows: let $R_1= \frac{R}{20}$. Then, assuming $R_1,\ldots,R_{n-1}$ are defined let $$
		R_{n}=\frac{R'_n}{20}, 
		$$ where $R'_n \leq R_{n-1}/2$ is the smallest radius for which some $p \in \mathcal{C}_N(\Omega)$ other than $\bar p$ intersects both 
		$S(x_0,R_{n-1}/2)$ and $S(x_0,R'_n)$. If no $p \in \mathcal{C}_N(\Omega)$ other than $\bar p$ intersects $S(x_0,R_{n-1}/2)$, we set $R'_n=R_{n-1}/2$. Then $R_{n}$ does not depend on $j$ and $R_n \to 0$ as $n \to \infty$. 
		
		The annuli $A_n := B(x_0, 2R_n) \setminus \overline{B}(x_0, R_n)$ and $ \pi_{\Omega_j}(A_n)$ are pairwise disjoint for a fixed $j=0,1,2...$. We denote by $\Gamma^j_n$ the family of all curves joining $\pi_{\Omega_j}(S(x_0,R_n))$ and $\pi_{\Omega_j}(S(x_0,2R_n))$ in $\hat{\Omega}_j$.
		
		By Proposition \ref{unibdd}, we know that $\Mod_\mathrm{T} \Gamma^j_n \leq M$, where $M$ does not depend on $j$ or $n$. Thus, if we fix an integer $N \geq 1$, then for every $1 \leq n \leq N$ there is an admissible function $\rho_n$ for the curve family $\Gamma^j_n$ such that
		\begin{equation}\label{4M}
			\int_{\Omega_{j}} \rho_n^2 \ d \mathcal{H}^2+\sum_{p \in \mathcal{C}(\Omega_j)} \rho_n(p)^2 \leq 2M.
		\end{equation}
		
		We claim that $\widetilde{\rho_N}:=\frac{\rho_1+\cdots+\rho_N}{N}$ is admissible for $\Gamma_{R, R_{N+1}}^j$. Indeed, each $\gamma \in \Gamma_{R, R_{N+1}}^j$ intersects $\pi_{\Omega_j}(S(x_0,R_i))$ for $i=1,2,\ldots, N$ but avoids $\bar p$, so we have

		$\begin{aligned}
			\int_\gamma \widetilde{\rho_N}\ ds+\sum_{p \in  \mathcal{C}(\Omega_j) \cap |\gamma|} \widetilde{\rho_N}(p) & \geq \sum_{i=1}^N \int_{\gamma\cap A_i} \widetilde{\rho_N}\ ds+ \sum_{i=1}^N \sum_{p \in C\left(\Omega_j\right) \cap|\gamma| \cap A_i} \widetilde{\rho_N}(p) \\ & \geq \sum_{i=1}^N \frac{1}{N}=1.
		\end{aligned}$
		
		Finally, we can estimate the modulus as follows:
		$$\begin{aligned} \Mod_\mathrm{T} \Gamma_{R, R_{N+1}}^j & \leq \int_{\Omega_j} \widetilde{\rho_N}^2\ d \mathcal{H}^2+\sum_{P \in \mathcal{C}(\Omega_j)} \rho_N(p)^2 \\ & \leq \frac{1}{N^2} \sum_{i=1}^N \left[\int_{\Omega_j} \rho_i^2\ d \mathcal{H}^2+\sum_{p \in \mathcal{C}(\Omega_j)}  \rho_i(p)^2\right] \\ & \leq \frac{1}{N^2} \cdot N \cdot 2M = \frac{2M}{N} \rightarrow 0 \ \text{as $N \rightarrow \infty$}.\end{aligned}$$
		Recall that the last inequality holds by \eqref{4M}.
		Since $R_N$ decreases to zero as $N \rightarrow \infty$ and $M,N$ do not depend on $j$, the claim follows. 
\end{proof}
	
	We now describe the setting where Proposition \ref{asyb} will be applied. The proof of the following lemma is an exercise.   
	\begin{lemma}\label{topo}
	Let $X$ be a metric space homeomorphic to $\mathbb{S}^2$. For every $\kappa>0$ there is an $\epsilon>0$ such that the following holds: if $a,b \in X$ and $d(a,b)<\epsilon$, then $a$ and $b$ can be joined by a curve $I=I_{a,b}$ in $\mathbb{S}^2$ with 
	$\diam(I)<\kappa$.    
	\end{lemma}

	Fix $\widetilde{p} \in \mathcal{C}(\Omega)$ and any Jordan curve $J\subset \Omega$ (recall that a \emph{Jordan curve} in $\Omega$ is a subset of $\Omega$ which is homeomorphic to $\mathbb{S}^1$). 
	Fix $0< \kappa <\dist(\widetilde{p}, J)$, and let $0<\epsilon \leq \kappa$ be as in Lemma \ref{topo}. Given $j=1,2,\ldots$ let $f_j: \Omega_{j} \to D_j$ be the quasiconformal map in Section \ref{existence}. 
	
	Next, fix $b \in \Omega \cap N_\epsilon(\widetilde{p})$, where $N_\epsilon(\widetilde{p})$ is the $\epsilon$-neighborhood of $\widetilde{p}$ in $X$. Choose any $a \in \partial\widetilde{p}$ whose distance to $b$ is less than $2 \epsilon$. By Lemma \ref{topo}, we may choose a curve $I=I_{a,b}$ in $X$ with diameter less than $\kappa$ joining $a$ and $b$. Notice that $I$ does not intersect $J$. 
			
	For every $j \geq 1$, denote
	$$\Lambda_j:=\{\text{curves in } \hat{\Omega}_j\setminus  \{\pi_{\Omega_j}(\widetilde{p})\} \text{ that join } \pi_{\Omega_j}(I) \text{ and } \pi_{\Omega_j}(J) \}, $$ 
	and notice that $\Lambda_j$ depends on the choices of $b$ and $a$. 
    We will apply the following circle domain estimate. Recall that $d_{\mathbb{S}^2}$ is the spherical metric on the standard $\mathbb{S}^2$. 
	\begin{proposition}\label{cir}
		There exists a homeomorphism $\psi:[0,\infty) \to [0,\infty)$, which depends on $\tilde p$ and $J$ but not on $j$, $b$, or $a$, such that 
		\begin{eqnarray*}
			\limsup_{j \to \infty} \Mod_\mathrm{T} \hat{f}_j(\Lambda_j) \geq \limsup_{j \to \infty} \psi(d_{\mathbb{S}^2}(f_j(b),\hat{f}_j(\widetilde{p}))). 
		\end{eqnarray*} 
      \end{proposition}
\begin{proof}
The proof given in \cite[Proposition 3.2 (1)]{esmayli2024conformal} can be adapted to our setting.  
\end{proof}


\section{Proof of Proposition \ref{unibdd}}\label{proofprop}

Recall that $X$ is a metric two-sphere which is upper Ahlfors $2$-regular with a constant $\alpha$. We fix a finitely connected $\tau$-cofat domain $\Omega \subset X$, $\bar p \in \hat{\Omega}$, $x_0 \in \bar p$, and $r>0$. We denote by $\Gamma_{x_0, r}$ the family of all curves joining $\pi_{\Omega}(S(x_0,r))$ and $\pi_{\Omega}(S(x_0,2r))$ in $\pi_{\Omega}({\overline{A}(x_0,r)})\subset \hat{\Omega}$. 

\begin{propositionx}[Proposition \ref{unibdd}]
	 There is an $M>0$ depending only on $\alpha$ and $\tau$ such that $\Mod_\mathrm{T} \Gamma_{x_0, r} \leq M$.
\end{propositionx} 

\begin{proof} [Proof of Proposition \ref{unibdd}]
	We define $\rho: \hat \Omega \to [0,\infty]$ such that 
		\begin{enumerate}
		\item [1) ] $\rho(x)= \frac{1}{r}$, \ if $x \in \Omega \cap A(x_0,r)$;
		\item [2) ] If $p \in \mathcal{C}(\Omega)$ and $p \cap A(x_0,r) \neq \emptyset$, \\ then 
		$\rho(p)=\left\{\begin{array}{l}1, \text { if } \operatorname{diam}(p) \geq \frac{r}{10} ; \\ \frac{\operatorname{diam}(p)}{r}, \text { if } \operatorname{diam}(p)<\frac{r}{10}; \end{array}\right.$
		\item[3)] $\rho(x)=0$ elsewhere. 
	\end{enumerate}
	
	We denote $$\mathcal{C}(x_0,r):=\{p\in \mathcal{C}(\Omega): p \cap A(x_0,r) \neq \emptyset\}, $$ and $$\mathscr{G}:= \left\{ p \in \mathcal{C}(x_0,r): \operatorname{diam}(p) \geq \frac{r}{10}\right\}. $$
\begin{lemma} \label{kro}
		$\rho$ is admissible for $\Gamma_{x_0, r}$.
\end{lemma}
\begin{proof}
	We fix a curve $\gamma \in \Gamma_{x_0, r}$. If $\gamma(t) =\bar p\in \mathscr{G}$ for some $t$, then $$\int_\gamma \rho\ d s+\sum_{p \in C(\Omega) \cap |\gamma|} \rho(p) \geq \rho(\bar p)=1.$$
	
	If $|\gamma| \cap  \mathscr{G} = \emptyset$, then the curve $\gamma$ can only "pass through" the complementary components with small diameters, see Figure \ref{xzj}. We denote $\gamma=\gamma_1 * \gamma_{p^1} * \gamma_2 * \gamma_{p^2} * \cdots * \gamma_n$, where $ \gamma_i$ is the curve in $\Omega$ and $ \gamma_{p^i}$ is the constant curve with value $\gamma_{p^i}=p^i \in \mathcal{C}(x_0,r)\setminus \mathscr{G}$.
	Then
	$$\begin{aligned} \int_\gamma \rho\ ds+\sum_{p \in C(\Omega) \cap |\gamma|} \rho(p)& =\sum_l \int_\gamma \rho\ d s+\sum_i \rho(p^l) \\ & \geq \frac{1}{r} \sum_l \ell (\gamma_l)+\frac{1}{r} \sum_l \operatorname{diam}(p^l) \\ & \geq \frac{1}{r} \cdot r=1, \end{aligned}$$ where we applied the triangle inequality.
\end{proof}

	\begin{figure}
	\begin{center} 
	
	\tikzset{
		pattern size/.store in=\mcSize, 
		pattern size = 5pt,
		pattern thickness/.store in=\mcThickness, 
		pattern thickness = 0.3pt,
		pattern radius/.store in=\mcRadius, 
		pattern radius = 1pt}
	\makeatletter
	\pgfutil@ifundefined{pgf@pattern@name@_8hvbgg91m}{
		\pgfdeclarepatternformonly[\mcThickness,\mcSize]{_8hvbgg91m}
		{\pgfqpoint{0pt}{-\mcThickness}}
		{\pgfpoint{\mcSize}{\mcSize}}
		{\pgfpoint{\mcSize}{\mcSize}}
		{
			\pgfsetcolor{\tikz@pattern@color}
			\pgfsetlinewidth{\mcThickness}
			\pgfpathmoveto{\pgfqpoint{0pt}{\mcSize}}
			\pgfpathlineto{\pgfpoint{\mcSize+\mcThickness}{-\mcThickness}}
			\pgfusepath{stroke}
	}}
	\makeatother
	
	
	\tikzset{
		pattern size/.store in=\mcSize, 
		pattern size = 5pt,
		pattern thickness/.store in=\mcThickness, 
		pattern thickness = 0.3pt,
		pattern radius/.store in=\mcRadius, 
		pattern radius = 1pt}
	\makeatletter
	\pgfutil@ifundefined{pgf@pattern@name@_tnr5eqov2}{
		\pgfdeclarepatternformonly[\mcThickness,\mcSize]{_tnr5eqov2}
		{\pgfqpoint{0pt}{-\mcThickness}}
		{\pgfpoint{\mcSize}{\mcSize}}
		{\pgfpoint{\mcSize}{\mcSize}}
		{
			\pgfsetcolor{\tikz@pattern@color}
			\pgfsetlinewidth{\mcThickness}
			\pgfpathmoveto{\pgfqpoint{0pt}{\mcSize}}
			\pgfpathlineto{\pgfpoint{\mcSize+\mcThickness}{-\mcThickness}}
			\pgfusepath{stroke}
	}}
	\makeatother
	
	
	\tikzset{
		pattern size/.store in=\mcSize, 
		pattern size = 5pt,
		pattern thickness/.store in=\mcThickness, 
		pattern thickness = 0.3pt,
		pattern radius/.store in=\mcRadius, 
		pattern radius = 1pt}
	\makeatletter
	\pgfutil@ifundefined{pgf@pattern@name@_8sjlv4t9o}{
		\pgfdeclarepatternformonly[\mcThickness,\mcSize]{_8sjlv4t9o}
		{\pgfqpoint{0pt}{-\mcThickness}}
		{\pgfpoint{\mcSize}{\mcSize}}
		{\pgfpoint{\mcSize}{\mcSize}}
		{
			\pgfsetcolor{\tikz@pattern@color}
			\pgfsetlinewidth{\mcThickness}
			\pgfpathmoveto{\pgfqpoint{0pt}{\mcSize}}
			\pgfpathlineto{\pgfpoint{\mcSize+\mcThickness}{-\mcThickness}}
			\pgfusepath{stroke}
	}}
	\makeatother
	
	
	\tikzset{
		pattern size/.store in=\mcSize, 
		pattern size = 5pt,
		pattern thickness/.store in=\mcThickness, 
		pattern thickness = 0.3pt,
		pattern radius/.store in=\mcRadius, 
		pattern radius = 1pt}
	\makeatletter
	\pgfutil@ifundefined{pgf@pattern@name@_jja2hf9ks}{
		\pgfdeclarepatternformonly[\mcThickness,\mcSize]{_jja2hf9ks}
		{\pgfqpoint{0pt}{-\mcThickness}}
		{\pgfpoint{\mcSize}{\mcSize}}
		{\pgfpoint{\mcSize}{\mcSize}}
		{
			\pgfsetcolor{\tikz@pattern@color}
			\pgfsetlinewidth{\mcThickness}
			\pgfpathmoveto{\pgfqpoint{0pt}{\mcSize}}
			\pgfpathlineto{\pgfpoint{\mcSize+\mcThickness}{-\mcThickness}}
			\pgfusepath{stroke}
	}}
	\makeatother
	
	
	\tikzset{
		pattern size/.store in=\mcSize, 
		pattern size = 5pt,
		pattern thickness/.store in=\mcThickness, 
		pattern thickness = 0.3pt,
		pattern radius/.store in=\mcRadius, 
		pattern radius = 1pt}
	\makeatletter
	\pgfutil@ifundefined{pgf@pattern@name@_57jjg7gxu}{
		\pgfdeclarepatternformonly[\mcThickness,\mcSize]{_57jjg7gxu}
		{\pgfqpoint{0pt}{-\mcThickness}}
		{\pgfpoint{\mcSize}{\mcSize}}
		{\pgfpoint{\mcSize}{\mcSize}}
		{
			\pgfsetcolor{\tikz@pattern@color}
			\pgfsetlinewidth{\mcThickness}
			\pgfpathmoveto{\pgfqpoint{0pt}{\mcSize}}
			\pgfpathlineto{\pgfpoint{\mcSize+\mcThickness}{-\mcThickness}}
			\pgfusepath{stroke}
	}}
	\makeatother
	
	
	\tikzset{
		pattern size/.store in=\mcSize, 
		pattern size = 5pt,
		pattern thickness/.store in=\mcThickness, 
		pattern thickness = 0.3pt,
		pattern radius/.store in=\mcRadius, 
		pattern radius = 1pt}
	\makeatletter
	\pgfutil@ifundefined{pgf@pattern@name@_mcj58ty64}{
		\pgfdeclarepatternformonly[\mcThickness,\mcSize]{_mcj58ty64}
		{\pgfqpoint{0pt}{-\mcThickness}}
		{\pgfpoint{\mcSize}{\mcSize}}
		{\pgfpoint{\mcSize}{\mcSize}}
		{
			\pgfsetcolor{\tikz@pattern@color}
			\pgfsetlinewidth{\mcThickness}
			\pgfpathmoveto{\pgfqpoint{0pt}{\mcSize}}
			\pgfpathlineto{\pgfpoint{\mcSize+\mcThickness}{-\mcThickness}}
			\pgfusepath{stroke}
	}}
	\makeatother
	
	
	\tikzset{
		pattern size/.store in=\mcSize, 
		pattern size = 5pt,
		pattern thickness/.store in=\mcThickness, 
		pattern thickness = 0.3pt,
		pattern radius/.store in=\mcRadius, 
		pattern radius = 1pt}
	\makeatletter
	\pgfutil@ifundefined{pgf@pattern@name@_rrxkxfw9m}{
		\pgfdeclarepatternformonly[\mcThickness,\mcSize]{_rrxkxfw9m}
		{\pgfqpoint{0pt}{-\mcThickness}}
		{\pgfpoint{\mcSize}{\mcSize}}
		{\pgfpoint{\mcSize}{\mcSize}}
		{
			\pgfsetcolor{\tikz@pattern@color}
			\pgfsetlinewidth{\mcThickness}
			\pgfpathmoveto{\pgfqpoint{0pt}{\mcSize}}
			\pgfpathlineto{\pgfpoint{\mcSize+\mcThickness}{-\mcThickness}}
			\pgfusepath{stroke}
	}}
	\makeatother
	
	
	\tikzset{
		pattern size/.store in=\mcSize, 
		pattern size = 5pt,
		pattern thickness/.store in=\mcThickness, 
		pattern thickness = 0.3pt,
		pattern radius/.store in=\mcRadius, 
		pattern radius = 1pt}
	\makeatletter
	\pgfutil@ifundefined{pgf@pattern@name@_v76phy6rz}{
		\pgfdeclarepatternformonly[\mcThickness,\mcSize]{_v76phy6rz}
		{\pgfqpoint{0pt}{-\mcThickness}}
		{\pgfpoint{\mcSize}{\mcSize}}
		{\pgfpoint{\mcSize}{\mcSize}}
		{
			\pgfsetcolor{\tikz@pattern@color}
			\pgfsetlinewidth{\mcThickness}
			\pgfpathmoveto{\pgfqpoint{0pt}{\mcSize}}
			\pgfpathlineto{\pgfpoint{\mcSize+\mcThickness}{-\mcThickness}}
			\pgfusepath{stroke}
	}}
	\makeatother
	
	
	\tikzset{
		pattern size/.store in=\mcSize, 
		pattern size = 5pt,
		pattern thickness/.store in=\mcThickness, 
		pattern thickness = 0.3pt,
		pattern radius/.store in=\mcRadius, 
		pattern radius = 1pt}
	\makeatletter
	\pgfutil@ifundefined{pgf@pattern@name@_9incqbmt9}{
		\pgfdeclarepatternformonly[\mcThickness,\mcSize]{_9incqbmt9}
		{\pgfqpoint{0pt}{-\mcThickness}}
		{\pgfpoint{\mcSize}{\mcSize}}
		{\pgfpoint{\mcSize}{\mcSize}}
		{
			\pgfsetcolor{\tikz@pattern@color}
			\pgfsetlinewidth{\mcThickness}
			\pgfpathmoveto{\pgfqpoint{0pt}{\mcSize}}
			\pgfpathlineto{\pgfpoint{\mcSize+\mcThickness}{-\mcThickness}}
			\pgfusepath{stroke}
	}}
	\makeatother
	
	
	\tikzset{
		pattern size/.store in=\mcSize, 
		pattern size = 5pt,
		pattern thickness/.store in=\mcThickness, 
		pattern thickness = 0.3pt,
		pattern radius/.store in=\mcRadius, 
		pattern radius = 1pt}
	\makeatletter
	\pgfutil@ifundefined{pgf@pattern@name@_pa3go6c0i}{
		\pgfdeclarepatternformonly[\mcThickness,\mcSize]{_pa3go6c0i}
		{\pgfqpoint{0pt}{-\mcThickness}}
		{\pgfpoint{\mcSize}{\mcSize}}
		{\pgfpoint{\mcSize}{\mcSize}}
		{
			\pgfsetcolor{\tikz@pattern@color}
			\pgfsetlinewidth{\mcThickness}
			\pgfpathmoveto{\pgfqpoint{0pt}{\mcSize}}
			\pgfpathlineto{\pgfpoint{\mcSize+\mcThickness}{-\mcThickness}}
			\pgfusepath{stroke}
	}}
	\makeatother
	
	
	\tikzset{
		pattern size/.store in=\mcSize, 
		pattern size = 5pt,
		pattern thickness/.store in=\mcThickness, 
		pattern thickness = 0.3pt,
		pattern radius/.store in=\mcRadius, 
		pattern radius = 1pt}
	\makeatletter
	\pgfutil@ifundefined{pgf@pattern@name@_yad89mtzt}{
		\pgfdeclarepatternformonly[\mcThickness,\mcSize]{_yad89mtzt}
		{\pgfqpoint{0pt}{-\mcThickness}}
		{\pgfpoint{\mcSize}{\mcSize}}
		{\pgfpoint{\mcSize}{\mcSize}}
		{
			\pgfsetcolor{\tikz@pattern@color}
			\pgfsetlinewidth{\mcThickness}
			\pgfpathmoveto{\pgfqpoint{0pt}{\mcSize}}
			\pgfpathlineto{\pgfpoint{\mcSize+\mcThickness}{-\mcThickness}}
			\pgfusepath{stroke}
	}}
	\makeatother
	
	
	\tikzset{
		pattern size/.store in=\mcSize, 
		pattern size = 5pt,
		pattern thickness/.store in=\mcThickness, 
		pattern thickness = 0.3pt,
		pattern radius/.store in=\mcRadius, 
		pattern radius = 1pt}
	\makeatletter
	\pgfutil@ifundefined{pgf@pattern@name@_jljzae12g}{
		\pgfdeclarepatternformonly[\mcThickness,\mcSize]{_jljzae12g}
		{\pgfqpoint{0pt}{-\mcThickness}}
		{\pgfpoint{\mcSize}{\mcSize}}
		{\pgfpoint{\mcSize}{\mcSize}}
		{
			\pgfsetcolor{\tikz@pattern@color}
			\pgfsetlinewidth{\mcThickness}
			\pgfpathmoveto{\pgfqpoint{0pt}{\mcSize}}
			\pgfpathlineto{\pgfpoint{\mcSize+\mcThickness}{-\mcThickness}}
			\pgfusepath{stroke}
	}}
	\makeatother
	
	
	\tikzset{
		pattern size/.store in=\mcSize, 
		pattern size = 5pt,
		pattern thickness/.store in=\mcThickness, 
		pattern thickness = 0.3pt,
		pattern radius/.store in=\mcRadius, 
		pattern radius = 1pt}
	\makeatletter
	\pgfutil@ifundefined{pgf@pattern@name@_tth6o0ls6}{
		\pgfdeclarepatternformonly[\mcThickness,\mcSize]{_tth6o0ls6}
		{\pgfqpoint{0pt}{-\mcThickness}}
		{\pgfpoint{\mcSize}{\mcSize}}
		{\pgfpoint{\mcSize}{\mcSize}}
		{
			\pgfsetcolor{\tikz@pattern@color}
			\pgfsetlinewidth{\mcThickness}
			\pgfpathmoveto{\pgfqpoint{0pt}{\mcSize}}
			\pgfpathlineto{\pgfpoint{\mcSize+\mcThickness}{-\mcThickness}}
			\pgfusepath{stroke}
	}}
	\makeatother
	
	
	\tikzset{
		pattern size/.store in=\mcSize, 
		pattern size = 5pt,
		pattern thickness/.store in=\mcThickness, 
		pattern thickness = 0.3pt,
		pattern radius/.store in=\mcRadius, 
		pattern radius = 1pt}
	\makeatletter
	\pgfutil@ifundefined{pgf@pattern@name@_aydivumoo}{
		\pgfdeclarepatternformonly[\mcThickness,\mcSize]{_aydivumoo}
		{\pgfqpoint{0pt}{-\mcThickness}}
		{\pgfpoint{\mcSize}{\mcSize}}
		{\pgfpoint{\mcSize}{\mcSize}}
		{
			\pgfsetcolor{\tikz@pattern@color}
			\pgfsetlinewidth{\mcThickness}
			\pgfpathmoveto{\pgfqpoint{0pt}{\mcSize}}
			\pgfpathlineto{\pgfpoint{\mcSize+\mcThickness}{-\mcThickness}}
			\pgfusepath{stroke}
	}}
	\makeatother
	
	
	\tikzset{
		pattern size/.store in=\mcSize, 
		pattern size = 5pt,
		pattern thickness/.store in=\mcThickness, 
		pattern thickness = 0.3pt,
		pattern radius/.store in=\mcRadius, 
		pattern radius = 1pt}
	\makeatletter
	\pgfutil@ifundefined{pgf@pattern@name@_ass2211hy}{
		\pgfdeclarepatternformonly[\mcThickness,\mcSize]{_ass2211hy}
		{\pgfqpoint{0pt}{-\mcThickness}}
		{\pgfpoint{\mcSize}{\mcSize}}
		{\pgfpoint{\mcSize}{\mcSize}}
		{
			\pgfsetcolor{\tikz@pattern@color}
			\pgfsetlinewidth{\mcThickness}
			\pgfpathmoveto{\pgfqpoint{0pt}{\mcSize}}
			\pgfpathlineto{\pgfpoint{\mcSize+\mcThickness}{-\mcThickness}}
			\pgfusepath{stroke}
	}}
	\makeatother
	\tikzset{every picture/.style={line width=0.75pt}} 
	
	\begin{tikzpicture}[x=0.75pt,y=0.75pt,yscale=-0.85,xscale=0.85]
		
		\draw   (71.12,242) .. controls (71.12,109.52) and (178.52,2.13) .. (311,2.13) .. controls (443.48,2.13) and (550.88,109.52) .. (550.88,242) .. controls (550.88,374.48) and (443.48,481.88) .. (311,481.88) .. controls (178.52,481.88) and (71.12,374.48) .. (71.12,242) -- cycle ;
		\draw   (196.38,242.88) .. controls (196.38,180.05) and (247.3,129.13) .. (310.13,129.13) .. controls (372.95,129.13) and (423.88,180.05) .. (423.88,242.88) .. controls (423.88,305.7) and (372.95,356.63) .. (310.13,356.63) .. controls (247.3,356.63) and (196.38,305.7) .. (196.38,242.88) -- cycle ;
		\draw  [line width=3] [line join = round][line cap = round] (311,241) .. controls (311,241) and (311,241) .. (311,241) ;
		\draw  [fill={rgb, 255:red, 0; green, 0; blue, 0}  ,fill opacity=0.25 ,even odd rule] (314,225) -- (328,251) -- (317,259) -- (310,264) -- (303,241) -- cycle ;
		\draw  [fill={rgb, 255:red, 0; green, 0; blue, 0}  ,fill opacity=0.25 ,even odd rule] (301,124) -- (312,135) -- (306,150) -- (287,159) -- (289,133) -- cycle ;
		\draw  [dash pattern={on 0.84pt off 2.51pt}]  (311,242) -- (477,70) ;
		\draw [line width=0.75]  [dash pattern={on 0.84pt off 2.51pt}]  (233,162) -- (311,242) ;
		\draw  [fill={rgb, 255:red, 0; green, 0; blue, 0}  ,fill opacity=0.25 ,even odd rule] (210,17) -- (277,87) -- (228,103) -- (200,123) -- (197,98) -- cycle ;
		\draw  [fill={rgb, 255:red, 0; green, 0; blue, 0}  ,fill opacity=0.25 ,even odd rule] (114.06,132.17) -- (162.11,140.49) -- (151,164) -- cycle ;
		\draw  [fill={rgb, 255:red, 0; green, 0; blue, 0}  ,fill opacity=0.25 ,even odd rule] (189,229) -- (219,251) -- (159,251) -- cycle ;
		\draw  [fill={rgb, 255:red, 0; green, 0; blue, 0}  ,fill opacity=0.25 ,even odd rule] (144,292) .. controls (144,285.92) and (148.92,281) .. (155,281) .. controls (161.08,281) and (166,285.92) .. (166,292) .. controls (166,298.08) and (161.08,303) .. (155,303) .. controls (148.92,303) and (144,298.08) .. (144,292) -- cycle ;
		\draw  [fill={rgb, 255:red, 0; green, 0; blue, 0}  ,fill opacity=0.25 ,even odd rule] (194.83,361.97) .. controls (185.97,352.87) and (182,342.37) .. (185.96,338.52) .. controls (189.92,334.66) and (200.31,338.92) .. (209.17,348.03) .. controls (218.03,357.13) and (222,367.63) .. (218.04,371.48) .. controls (214.08,375.34) and (203.69,371.08) .. (194.83,361.97) -- cycle ;
		\draw  [fill={rgb, 255:red, 0; green, 0; blue, 0}  ,fill opacity=0.25 ,even odd rule] (298,394) .. controls (318,384) and (306,413) .. (326,403) .. controls (346,393) and (358,413) .. (352,422) .. controls (346,431) and (318,453) .. (298,423) .. controls (278,393) and (278,404) .. (298,394) -- cycle ;
		\draw [fill={rgb, 255:red, 0; green, 0; blue, 0}  ,fill opacity=0.25 ,even odd rule] (421.24,241.63) -- (430.81,235.02) -- (422.21,253.22) -- cycle ;
		\draw  [fill={rgb, 255:red, 0; green, 0; blue, 0}  ,fill opacity=0.25 ,even odd rule] (452.06,186.86) .. controls (452.06,183.01) and (455.18,179.89) .. (459.03,179.89) .. controls (462.88,179.89) and (466,183.01) .. (466,186.86) .. controls (466,190.7) and (462.88,193.82) .. (459.03,193.82) .. controls (455.18,193.82) and (452.06,190.7) .. (452.06,186.86) -- cycle ;
		\draw [color={rgb, 255:red, 74; green, 144; blue, 226 }  ,draw opacity=1 ]   (423.88,242.88) .. controls (446.88,242.88) and (424.03,269.97) .. (450.03,256.97) ;
		\draw  [fill={rgb, 255:red, 0; green, 0; blue, 0}  ,fill opacity=0.25 ,even odd rule] (462.03,247.97) -- (466,261) -- (466,269) -- (459,262) -- (450.03,256.97) -- cycle ;
		\draw [color={rgb, 255:red, 74; green, 144; blue, 226 }  ,draw opacity=1 ]   (466,269) .. controls (489,269) and (385,312) .. (423,313) ;
		\draw  [fill={rgb, 255:red, 0; green, 0; blue, 0}  ,fill opacity=0.25 ,even odd rule] (423,313) .. controls (432,311) and (428,309) .. (436,309) .. controls (444,309) and (443,317) .. (437,326) .. controls (431,335) and (431,345) .. (424,337) .. controls (417,329) and (414,315) .. (423,313) -- cycle ;
		\draw  [fill={rgb, 255:red, 0; green, 0; blue, 0}  ,fill opacity=0.25 ,even odd rule] (469.47,346.7) -- (463.94,354.31) -- (455,351.4) -- (455,342) -- (463.94,339.09) -- cycle ;
		\draw [color={rgb, 255:red, 74; green, 144; blue, 226 }  ,draw opacity=1 ]   (437,326) .. controls (460,326) and (429,355) .. (455,342) ;
		\draw  [fill={rgb, 255:red, 0; green, 0; blue, 0}  ,fill opacity=0.25 ,even odd rule] (56.84,217.91) .. controls (55.13,205.88) and (96.57,190.03) .. (149.41,182.5) .. controls (202.26,174.97) and (246.48,178.61) .. (248.19,190.64) .. controls (249.91,202.67) and (208.46,218.53) .. (155.62,226.06) .. controls (102.78,233.59) and (58.56,229.94) .. (56.84,217.91) -- cycle ;
		\draw  [fill={rgb, 255:red, 0; green, 0; blue, 0}  ,fill opacity=0.25 ,even odd rule] (375.36,75.93) -- (381.09,120.3) -- (350.64,153.07) -- (344.91,108.7) -- cycle ;
		\draw [color={rgb, 255:red, 74; green, 144; blue, 226 }  ,draw opacity=1 ]   (463.94,354.31) .. controls (503.94,324.31) and (475,398) .. (515,368) ;
		
		\draw (305,243) node [anchor=north west][inner sep=0.75pt]   [align=left] {$x_0$};
		\draw (268,185) node [anchor=north west][inner sep=0.75pt]   [align=left] {$r$};
		\draw (425,125) node [anchor=north west][inner sep=0.75pt]   [align=left] {$2r$};
		\draw (299,81) node [anchor=north west][inner sep=0.75pt]   [align=left] {....};
		\draw (239,381) node [anchor=north west][inner sep=0.75pt]   [align=left] {....};
		\draw (299,210.5) node [anchor=north west][inner sep=0.75pt]   [align=left] {$\bar p$};
		\draw (417,286) node [anchor=north west][inner sep=0.75pt]   [align=left] {.};
		\draw (423,280) node [anchor=north west][inner sep=0.75pt]   [align=left] {.};
		\draw (420,283) node [anchor=north west][inner sep=0.75pt]   [align=left] {.};
		\draw (492,350) node [anchor=north west][inner sep=0.75pt]   [align=left] {$\gamma_1$};
		\draw (469,325) node [anchor=north west][inner sep=0.75pt]   [align=left] {$p^1$};
		\draw (444,317) node [anchor=north west][inner sep=0.75pt]   [align=left] {$\gamma_2$};
		\draw (395,322) node [anchor=north west][inner sep=0.75pt]   [align=left] {$p^2$};
		\draw (405,289) node [anchor=north west][inner sep=0.75pt]   [align=left] {$\gamma_3$};
		\draw (432.81,238.02) node [anchor=north west][inner sep=0.75pt]   [align=left] {$\gamma_n$};
		\draw (401.81,238.02) node [anchor=north west][inner sep=0.75pt]   [align=left] {$p^n$};
		\draw (425,269) node [anchor=north west][inner sep=0.75pt]   [align=left] {$\gamma_{n-1}$};
		\draw (466,245) node [anchor=north west][inner sep=0.75pt]   [align=left] {$p^{n-1}$};

	\end{tikzpicture}
    \end{center} 
    \caption{Part of a sample curve $\gamma=\gamma_1 * p^1 * \cdots * p^n$ that satisfies $\gamma(t) \notin \mathscr{G}$, each $\operatorname{diam}(p^i) < \frac{r}{10}$.}
    \label{xzj}
    \end{figure}
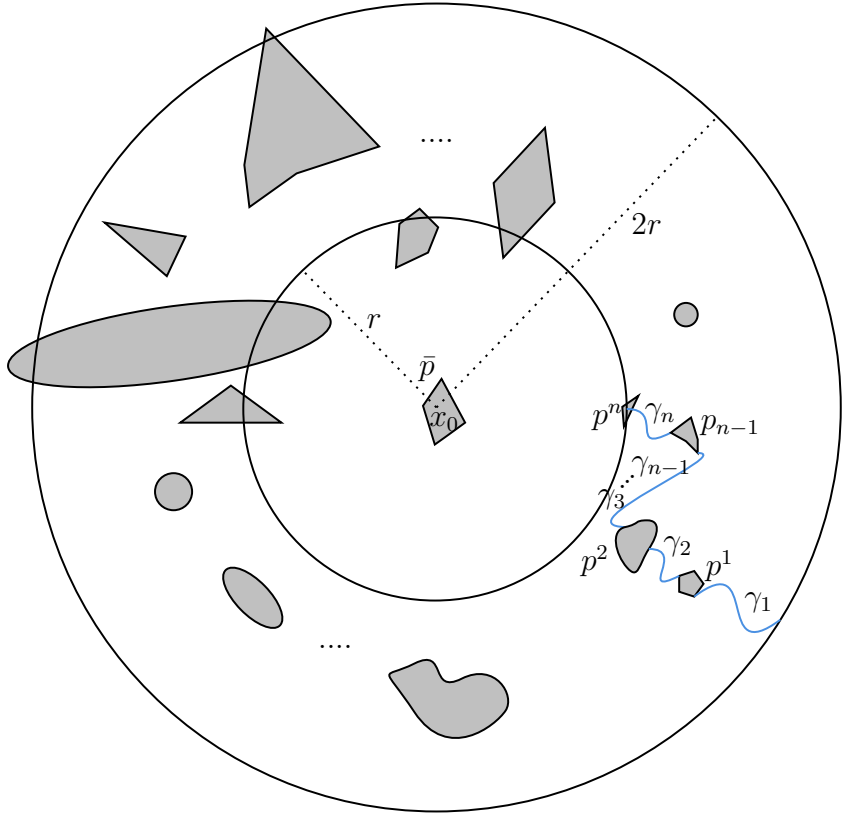

We denote the cardinality of a set $A$ by $\#A$. 
\begin{lemma}\label{number}
		$\# \mathscr{G} \leq N$, where $N>0$ depends only on $\alpha$ and $\tau$.
\end{lemma}
	
\begin{proof} 
Given $p \in \mathscr{G}$, choose $a \in  p \cap A(x_0,r) $. Since $\operatorname{diam}(p) \geq \frac{r}{10}$, we have $p \not \subset B(a, r/40)$. By the $\tau$-fatness of $p$ and 
the upper Ahlfors 2-regularity of $X$, we have 
	\begin{equation}\label{num}
		\tau \cdot(r/40)^2  \leq \mathcal{H}^2(p \cap B(a, r/40))   \leq \mathcal{H}^2(B(x_0, 3 r))  \leq \alpha \cdot 9 r^2.
	\end{equation}
	Summing \eqref{num} over for all $p \in \mathscr{G}$ yields 
	$$\frac{\tau r^2}{1600} \cdot (\# \mathscr{G}) \leq \sum_{p \in  \mathscr{G}} \mathcal{H}^2(p \cap B(a, r/40)) \leq \alpha \cdot 9r^2, $$
	so we have the uniformly upper bound of $\# \mathscr{G}:$
	$$\# \mathscr{G} \leq \frac{14400\alpha}{\tau} =: N(\alpha, \tau). $$	
\end{proof}	
	
	Finally, we prove the uniform boundedness of $\Mod_\mathrm{T} \Gamma_{x_0, r}$.
	Since $\rho$ is an admissible function for $\Gamma_{x_0, r}$ by Lemma \ref{kro}, we have
	\begin{equation}\label{unif}
		\Mod_\mathrm{T} \Gamma_{x_0, r}\leq \int_{\Omega} \rho^2 d \mathcal{H}^2+\sum_{p \in \mathcal{C}(\Omega)} \rho(p)^2.
	\end{equation}
	
	We divide the right-hand side of \eqref{unif} into two terms and estimate them separately. First, we have
	\begin{equation}\label{12}
		\begin{aligned} \int_{\Omega} \rho^2 d \mathcal{H}^2 & =\int_{\Omega\cap A(x_0,r)} \frac{1}{r^2} d \mathcal{H}^2 = \frac{1}{r^2} \cdot \mathcal{H}^2 (\Omega\cap A(x_0,r)) \\ & \leq \frac{1}{r^2} \mathcal{H}^2 (B\left(x_0, 3 r\right)) \leq 9\alpha,  \end{aligned}
	\end{equation}
	by Lemma \ref{number} and the upper $2$-regularity of $X$. Secondly, we can estimate the sum in \eqref{unif} as follows:
	
	\begin{equation}\label{23}
		\begin{aligned} \sum_{p \in \mathcal{C}(\Omega)} \rho(p)^2 & =\sum_{p \in \mathscr{G}} \rho(p)^2+\sum_{p \notin \mathscr{G}} \rho(p)^2 = \# \mathscr{G}+\sum_{p \notin \mathscr{G}} \frac{\operatorname{diam}(p)^2}{r^2} \\ & \leq N+\sum_{p \notin \mathscr{G}} \frac{\mathcal{H}^2(p)}{\tau \cdot r^2} \leq N+\frac{\mathcal{H}^2(B(x_0,3r))}{\tau \cdot r^2} \\ & \leq N+\frac{9\alpha}{\tau}. \end{aligned}
	\end{equation}
	By \eqref{unif}, \eqref{12}, and \eqref{23}, we have $\Mod_\mathrm{T} \Gamma_{x_0, r} \leq 9\alpha +N+\frac{9\alpha}{\tau}$. 
\end{proof}


\section{Proof of Theorem \ref{bz}}\label{proofthm}

In this section, we prove that the $\frac{\pi}{2}$-QC map $f$ constructed in Section \ref{existence} preserves the classes of complementary components and the image is a circle domain on $\mathbb{S}^2$. 

We first recall that $f: \Omega \to D$ is the locally uniform limit of the sequence of $\frac{\pi}{2}$-QC maps $(f_j)_j$ given in Section \ref{existence}. Moreover, if $\ell=1,2,\ldots$ then 
$p_\ell \in \mathcal{C}_N(\Omega)$ and 
\begin{equation} \label{faith}
\hat{f}_j(p_\ell)\to q_\ell \quad \text{as} \quad j \to \infty 
\end{equation} 
in the Hausdorff sense. We prove Theorem \ref{bz} in three steps:

\textbf{Step 1: From point to point.} To prove \eqref{pro1}, we assume that $\widetilde{p}= \{a\} \in \mathcal{C}_P(\Omega)$ and suppose towards a contradiction that 
$\hat{f}(\widetilde{p}) \in \mathcal{C}_N(D)$. We fix a Jordan curve $J \subset \Omega$, and set 
$$R:=\operatorname{dist}(a,J)>0.$$

Then we can find a $c>0$ and a sequence $(b_k)_{k=1}^\infty$ of points in $\Omega$ converging to $a$, such that for every $k$, we have
\begin{equation} \label{worsor}
	\limsup_{j \to \infty}d_{\mathbb{S}^2}(f_j(b_k),f_j(a)) \geq c >0.  
\end{equation}   
Let $(\kappa_k)_{k=1}^\infty$ be a sequence such that
$$d(a,b_k)<\kappa_k<R \quad \text{for every } k,\,\,
\text{and} \,\,
\kappa_k\to 0 \,\, \text{as} \,\,k\to\infty.$$

Moreover, let $I_k \subset B(a,\kappa_k)$, $k=1,2,\ldots$, be a curve in $X$ joining $a$ to $b_k$. Since $\kappa_k<R$, each $I_k$ is disjoint from $J$.
For $j,k=1,2,\ldots$, let $\Lambda_{j,k}$ be the family of curves in $\hat{\Omega}_j\setminus  \{\pi_{\Omega_j}(a)\}$
joining $\pi_{\Omega_j}(I_k)$ and $\pi_{\Omega_j}(J)$.

By \eqref{worsor} and Proposition \ref{cir}, there is $M=M(c)>0$ such that  
\begin{equation} \label{positive} 
	\limsup_{j \to \infty} \Mod_\mathrm{T} \hat{f}_j(\Lambda_{j,k}) \geq M \quad \text{for every } k.  
\end{equation}

On the other hand, every curve $\gamma \in \Lambda_{j,k}$ intersects both $\pi_{\Omega_j}(S(a,\kappa_k))$ and $\pi_{\Omega_j}(S(a,R))$. Therefore, 
if we denote the family of curves joining $\pi_{\Omega_j}(S(a,\kappa_k))$ and $\pi_{\Omega_j}(S(a,R))$ in $\hat{\Omega}_j$ by $\Gamma_{R, \epsilon}^j$, we have 

\begin{equation}\label{mbj}
	\Mod_\mathrm{T} (\Lambda_{j,k}) \leq \Mod_\mathrm{T} (\Gamma_{R, \kappa_k}^j).
\end{equation}

By \eqref{mbj}, Proposition \ref{asyb}, and the assumption $\widetilde{p} \in \mathcal{C}_P(\Omega)$, we have 
$$\limsup_{j \to \infty} \Mod_\mathrm{T} (\Lambda_{j,k}) \leq \zeta(k), \quad \text{where } \zeta(k) \to 0 \text{ as } k \to \infty.$$
This contradicts \eqref{positive} and Theorem \ref{qcinvariance} (quasi-invariance of the transboundary modulus), so $\hat{f}(\widetilde{p})$ must be a point-component.

\textbf{Step 2: The image is a circle domain.} To prove the first claim of \eqref{pro2}, i.e., that $\hat{f}(p_\ell)=q_\ell$ for all the non-trivial components in \eqref{faith}, we use the idea in the proof of \eqref{pro1}.

From a quasiconformal version of Carath\'eodory's kernel convergence theorem \cite[Lemma 2.14]{ntalampekos2023conformal}, we know that $q_\ell \subset \hat{f}(p_\ell)$. 
We argue by contradiction and suppose that $q_\ell \subsetneq \hat{f}(p_\ell)$. Then there are a $c>0$ and a sequence $(b_k)_k$ of points in $\Omega$ such that $d(b_k,p_\ell) \to 0$ as $k \to \infty$ and
\begin{equation} \label{dist12} 
	\limsup_{j \to \infty}d_{\mathbb{S}^2}(f_j(b_k),\hat{f}_j(p_\ell)) \geq c >0.  
\end{equation}
We fix a Jordan curve $J\subset\Omega$. After passing to a subsequence if necessary, we may assume that the sequence $(b_k)_k$ converges to some point $a\in \partial p_{\ell}$. As in Step 1, we can choose curves $I_k$, $k=1,2,\ldots$, in $X$ joining $a$ and $b_k$, such that $\diam(I_k) \to 0$ as $k \to \infty$. We then denote by $\Lambda_{j,k}$ the family of all curves joining $\pi_{\Omega_j}(I_k)$ and $\pi_{\Omega_j}(J)$ in $\hat{\Omega}_j\setminus  \{\pi_{\Omega_j}(a)\}$. 

Combining \eqref{dist12} with Proposition \ref{cir} shows that there is $M>0$ for which 
\begin{equation} \label{positive12'} 
	\limsup_{j \to \infty} \Mod_\mathrm{T} \hat{f}_j(\Lambda_{j,k}) \geq M \quad \text{for every } k.
\end{equation} 
On the other hand, since \eqref{mbj} remains valid and $\diam(I_k) \to 0$, Proposition \ref{asyb} yields  
$$\limsup_{j \to \infty} \Mod_\mathrm{T} (\Lambda_{j,k}) \leq \zeta(k), \quad \text{where } \zeta(k) \to 0 \text{ as } k \to \infty.$$
This contradicts \eqref{positive12'} and Theorem \ref{qcinvariance} (quasi-invariance of the transboundary modulus). It follows that $\hat{f}(p_\ell)$ is the disk $q_\ell$. Combining Steps 1 and 2, we conclude that $D$ is a circle domain.

\textbf{Step 3: From non-trivial to non-trivial.}
Finally, we need to show the second claim of \eqref{pro2} i.e., $\hat{f}(\mathcal{C}_N(\Omega)) \subset \mathcal{C}_N(D)$.

Fix $p_\ell \in \mathcal{C}_N(\Omega)$ and a Jordan curve $J \subset \Omega$. Then $p_\ell \in \mathcal{C}_N(\Omega_j)$ and $J \subset \Omega_j$ for every $j \geq \ell$. 
Denote by $\Lambda_j'$ the family of all curves joining $\pi_{\Omega_j}(p_\ell)$ and $\pi_{\Omega_j}(J)$ in $\hat{\Omega}_j$.
Towards a contradiction, we assume that $ \hat{f}(p_\ell)=q_\ell$ is a point. We claim that 
\begin{equation} \label{positive1234} 
	\limsup_{j \to \infty} \Mod_\mathrm{T} \hat{f}_j(\Lambda_j') =0. 
\end{equation}
We need the following lemma to prove \eqref{positive1234}. 

\begin{lemma}\label{separate}
	For every $R>0$, there exist an index $j_R\geq 1$ and a radius $0<r_R<R$, such that for every $j>j_R$ and every $q\in \mathcal{C}(D_j)$ with $q \cap S(q_\ell, R) \neq \emptyset$, we have $q\cap S(q_\ell, r_R)= \emptyset$.
\end{lemma}

\begin{proof}
	We argue by contradiction. Suppose that there exist $R_0>0$, a subsequence $(f_{j_i})_i$ of $(f_j)_j$, and a sequence of complementary components $p^i\in \mathcal{C}_N(\Omega_{j_i})$, such that for each $q^i:=\hat{f}_{j_i}(p^i) \in \mathcal{C}(D_{j_i})$ satisfying $q^i \cap S(q_\ell, R_0) \neq \emptyset$, we have 
	\begin{equation}\label{j1}
		\lim_{i \to \infty}d_{\mathbb{S}^2}(q_\ell ,q^i)=0.
	\end{equation}
	
	After taking a subsequence if necessary, we may assume that $(p^i)_i$ converges to some $\bar p \in \mathcal{C}(\Omega)$ in the Hausdorff sense. Let $\bar q:=\hat{f}(\bar p)$. Then, by Carath\'eodory's kernel convergence theorem, for every $\epsilon>0$, there exist a $\delta>0$ and an index $i_\epsilon \geq 1$ such that 
	$$\hat{f}_{j_i}\big(N_\delta(\bar p)\big) \subset N_{\epsilon}(\bar q)\quad \text{for all $i>i_\epsilon$},$$
	where $N_\delta(\bar p)$ denotes the $\delta$-neighborhood of $\bar p$ in $X$, and $N_\epsilon(\bar q)$ denotes the $\epsilon$-neighborhood of $\bar q$ in $\mathbb{S}^2$.
	
	On the other hand, by the Hausdorff convergence of $(p^i)_i$, there exists an index $i_\delta\geq 1$ such that $p^i \in N_\delta(\bar p)$ for every $i>i_\delta$. It follows that 
	$$d_{\mathbb{S}^2}(\bar q ,q^i)\to 0\quad \text{as $i\to \infty$}. $$
	Combining this with \eqref{j1}, we conclude that $\bar q=q_\ell$. However, this contradicts the assumption that $q^i \cap S(q_\ell, R_0) \neq \emptyset$ for every $i\geq 1$. The proof is complete.
\end{proof}

We continue to prove \eqref{positive1234}. We construct a sequence of radii  $(r_n)_n$ as in the proof of Proposition \ref{asyb}: Let $$r_0:= d_{\mathbb{S}^2}(f(J), q_\ell)>0.$$ Since $f_j \to f$ locally uniformly in $\Omega$, we may assume that $$d_{\mathbb{S}^2}(f_j(J), q_\ell)>r_{0}/2$$ for all sufficiently large indices $j$.
 We first define $r_1:= r_0/20$. Assuming that $r_1>r_2>\ldots>r_{n-1}>0$ have been chosen, we apply Lemma \ref{separate} with $R= r_{n-1}/2$ to obtain an index $j_n \geq 1$ and 
 a radius $0<\widetilde{r}_n < r_{n-1}/2$ such that for every $j>j_n$, and every $q\in \mathcal{C}(D_j)$ with $q \cap S(q_\ell, r_{n-1}/2) \neq \emptyset$, we have $q\cap S(q_\ell, \widetilde{r}_n)= \emptyset$. We then define
$$r_n:=\frac{\widetilde{r}_n}{20}.$$

Note that the sequence $(r_n)_n$ is independent of $j$, and satisfies $r_n\to 0$ as $n\to \infty$. For each $n\geq 1$, define the annulus $$\widetilde{A}_n := B(q_\ell, 2r_n) \setminus \overline{B}(q_\ell, r_n) \subset \mathbb{S}^2.$$ 
Then for any fixed  $j$, the sets $ \pi_{D_j}(\widetilde{A}_n)$ are pairwise disjoint. We can now follow the argument in the proof of Proposition \ref{asyb} to conclude that \eqref{positive1234} holds. 
The following lemma and the quasi-invariance of the transboundary modulus will lead us to a contradiction with \eqref{positive1234}.  

\begin{lemma}\label{pm} 
We have $\Mod_\mathrm{T} \Lambda_j' \geq c > 0$, where $c$ is independent of $j$.
\end{lemma}

\begin{proof}
	Let $\gamma$ be a curve in $X$ joining $J$ and $\partial p_\ell$, and let $m(x):= \dist(x, |\gamma|)$. Then $m$ is 1-Lipschitz. 
	
	Fix a point $\widetilde{b} \in X$ that is not in the same component of $X \setminus J$ as $p_\ell$. Then there exists a $\delta':=\delta'(\widetilde{b})>0$ such that $E_t := m^{-1}(t)$ separates $|\gamma|$ and $\{\widetilde{b}\}$ when $0<t<\delta'$.
	Since $X$ is homeomorphic to $\mathbb{S}^2$, there exists a continuum $\eta_t \subset E_t$ which also separates $|\gamma|$ and $\{\widetilde{b}\}$ in $X$; see \cite[Ch. 2, Lemma 5.20]{Wil79}. 
	
	Since $\diam(p_\ell)>0$, there exists a $0<\delta<\delta'$ such that $\eta_t$ intersects both $\partial p_\ell$ and $J$ for every $0<t<\delta$. Applying the Coarea inequality (\cite{esmayli2021coarea}) to the 1-Lipschitz function $m$ and $g=1$, we have $\mathcal{H}^1(\eta_t) <\infty$ for almost every $0<t<\delta$. Thus, by (\cite[Proposition 15.1]{semmes1996finding}) there exists an injective 1-Lipschitz curve $\gamma_t$ such that $| \gamma_t| \subset \eta_t$ and $\gamma_t$ connects $\partial p_\ell$ and $J$ in $X$, so we have $\hat{\gamma}_t:= \pi_{\Omega_j}(\gamma_t)\in \Lambda_j'$.
	
	Now we can use the curve $\hat{\gamma}_t$ to find a lower bound for $\Mod_\mathrm{T} \Lambda_j'$. Let $\rho$ be admissible for $\Lambda_j'$. Since $\hat{\gamma}_t \in \Lambda_j'$ for almost every $0<t<\delta$ and the restriction of $\hat{\gamma}_t$ to $\Omega_{j}$ is injective, applying the Coarea inequality yields $$
	\begin{aligned}
		& \delta \leq \int_0^\delta \left[\int_{\hat{\gamma}_t } \rho\, d s+\sum_{p \in \mathcal{C}(\Omega_j) \cap | \hat{\gamma}_t| } \rho (p)\right] dt \\
		& \leq \int_0^\delta \int_{E_t \cap \Omega_j} \rho \, d \mathcal{H}^1 dt+\sum_{p \in \mathcal{C}(\Omega_{j})} \rho(p) \cdot \diam(p) \\
		& \leq \frac{4}{\pi} \left[\int_{\Omega_j} \rho \, d \mathcal{H}^2+\sum_{p \in \mathcal{C}(\Omega_{j})} \rho(p) \cdot \diam(p)\right] =: \frac{4I}{\pi}. \\ 
	\end{aligned}  
	$$
Moreover, applying Hölder's inequality yields 
\begin{equation} \label{rasta}
	I \leq \left[ \mathcal{H}^2(\Omega_j) +  \sum_{p \in \mathcal{C}(\Omega_{j})} \diam(p)^2 \right]^{\frac{1}{2}} 
		\left[\int_{\Omega_j} \rho^2 \, d\mathcal{H}^2 + \sum_{p \in \mathcal{C}(\Omega_{j})} \rho(p)^2\right]^{\frac{1}{2}}.  
\end{equation} 
	Since $\mathcal{H}^2(X)<\infty$ and $\diam(p)^2 \leq \tau^{-1} \mathcal{H}^2(p)$ for every $p \in \mathcal{C}(\Omega_j)$ by cofatness, it follows that 
\begin{equation} \label{nasta}	 
		\mathcal{H}^2(\Omega_j) +  \sum_{p \in \mathcal{C}(\Omega_{j})} \diam(p)^2  \leq M < \infty, 	
\end{equation} 
	where $M$ does not depend on $j$. Since the estimates hold for all admissible functions $\rho$, the claim follows by combining \eqref{rasta} and \eqref{nasta}. 
\end{proof}

Lemmas \ref{pm} and \ref{qcinvariance} (quasi-invariance of transboundary modulus) contradict \eqref{positive1234}. We conclude that $\hat{f}$ maps non-trivial components to non-trivial components. The proof of Theorem \ref{main} is complete. 

\begin{remark}\label{diam2yx}
We demonstrate how the above arguments can be applied to prove Theorem \ref{dimi2}, leaving the details to the reader. Suppose that $\Omega_2 \in \mathcal{F}_2(X)$, and let $f:\Omega_2 \to D$ 
be a conformal map onto a circle domain $D \subset \mathbb{S}^2$. Fix $p \in \mathcal{C}_N(\Omega_2)$ and a Jordan curve $J \subset \Omega_2$, and let 
$\Gamma$ be the family of curves joining $\pi_{\Omega_2}(p)$ and $\pi_{\Omega_2}(J)$ in $\hat{\Omega}_2$. The proof of Lemma \ref{pm} can be adapted for $\ell^2$-domains to 
show that $\Mod_\mathrm{T} \Gamma >0$. 

On the other hand, since the circle domain $D$ is cofat, the proof of Proposition \ref{asyb} can be modified to show that if $\hat{f}(p) \in \mathcal{C}_P(D)$, then $\Mod_\mathrm{T} \hat{f}(\Gamma)=0$; 
here we also need the countable connectedness of $D$. We arrive at a contradiction with the conformal invariance of the transboundary modulus (Theorem \ref{qcinvariance}). Theorem \ref{dimi2} 
follows. 
\end{remark}




\section{Proof of Theorem \ref{nontop2}}
In this section we demonstrate the sharpness of Theorem \ref{dimi2} and Remark \ref{diam2yx}. Recall that $\mathcal{F}_\alpha(X)$ is the collection of $\ell^\alpha$-domains on $X$, i.e., the countably connected subdomains $\Omega_\alpha$ of $X$ satisfying  $$\sum_{p \in \mathcal{C}(\Omega_\alpha)} \operatorname{diam}(p)^\alpha<\infty. $$  

\begin{theoremx}[Theorem \ref{nontop2}]
For every $\alpha>2$ there is a domain $U_\alpha \in \mathcal{F}_\alpha(\mathbb{S}^2)$, such that $\overline{\mathbb{D}} \in \mathcal{C}_N(U_\alpha)$ but $\hat{f}(\overline{\mathbb{D}} ) \in \mathcal{C}_P(D)$ for every conformal homeomorphism $f: U_\alpha \to D$ onto a circle domain $D \subset \mathbb{S}^2$.
\end{theoremx}
     Note that such a conformal homeomorphism $f$ always exists by the He-Schramm theorem \cite[Theorem 0.1]{he1993fixed}.
\begin{proof}[Proof of Theorem \ref{nontop2}]
   Fix $\alpha>2$. We construct a countably connected domain $U_\alpha \subset \mathbb{S}^2$ whose complementary components include the closed unit disk $\overline{\mathbb{D}}$ as well as countably many subarcs of concentric circles "winding around" $\overline{\mathbb{D}}$, such that $\hat{f}(\overline{\mathbb{D}}) \in \mathcal{C}_P(D)$ for every conformal map $f: U_\alpha \to D$ onto a circle domain $D$. 
   
   We now describe the elements of $\mathcal{C}(U_\alpha)$. First, denote 
   $$p_{0,0}:=\overline{\mathbb{D}}.$$
   Then, given an index $k\geq 1$ and $0\leq j \leq k$, let
    $$p_{k, j}:=\Big\{r_k e^{i \theta_j}: r_k=1+\frac{1}{2^{k-1}},\; \theta_j \in[\alpha_{k, j}, \beta_{k, j}] \Big\},$$
   where 
   \begin{eqnarray*}
   	\left\{\begin{array}{l}\alpha_{k, j}:=\frac{2(j-1) \pi}{k+1}+\theta_k, \\ \beta_{k, j}:=\frac{2 j \pi}{k+1}-\theta_k,\end{array}
   	\right. 
   \end{eqnarray*} 
   and \begin{equation}\label{thetak}
   	\theta_k:=\frac{2 \pi}{e^{2^k} \cdot 2025^k}.
   \end{equation}
   We denote the endpoints of each $p_{k,j}$ by
   $$a_{k, j}:=r_k e^{i \alpha_{k, j}}, \quad b_{k, j}:=r_k e^{i \beta_{k, j}}.$$
   Finally, let $U_\alpha$ be the domain for which 
   $$
   \mathcal{C}(U_\alpha) := \{ p_{k,j} : k=1,2,\ldots, \, j=0,1,\ldots,k \}.
   $$
   Figure \ref{ua} shows the complementary components $p_{0,0}, p_{1,j}, p_{2,j}$ and $p_{3,j}$. 
   
   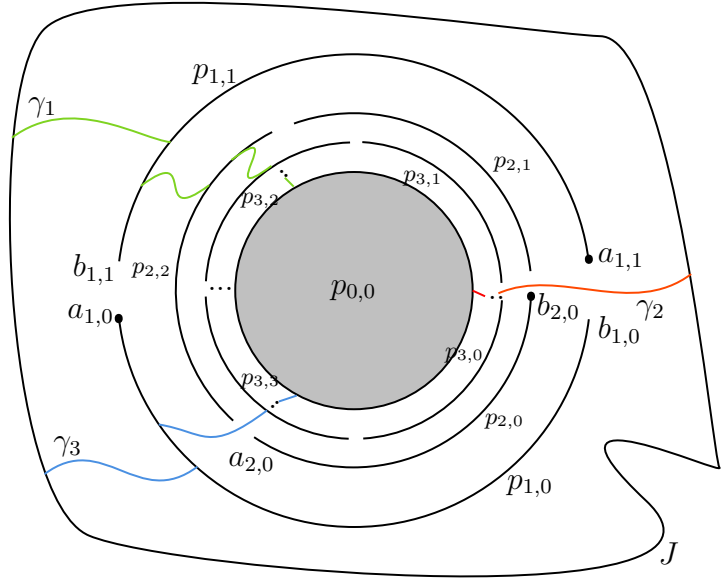
\begin{figure}

   	
   	\tikzset{
   		pattern size/.store in=\mcSize, 
   		pattern size = 5pt,
   		pattern thickness/.store in=\mcThickness, 
   		pattern thickness = 0.3pt,
   		pattern radius/.store in=\mcRadius, 
   		pattern radius = 1pt}
   	\makeatletter
   	\pgfutil@ifundefined{pgf@pattern@name@_1sxsnul90}{
   		\pgfdeclarepatternformonly[\mcThickness,\mcSize]{_1sxsnul90}
   		{\pgfqpoint{0pt}{-\mcThickness}}
   		{\pgfpoint{\mcSize}{\mcSize}}
   		{\pgfpoint{\mcSize}{\mcSize}}
   		{
   			\pgfsetcolor{\tikz@pattern@color}
   			\pgfsetlinewidth{\mcThickness}
   			\pgfpathmoveto{\pgfqpoint{0pt}{\mcSize}}
   			\pgfpathlineto{\pgfpoint{\mcSize+\mcThickness}{-\mcThickness}}
   			\pgfusepath{stroke}
   	}}
   	\makeatother
   	\tikzset{every picture/.style={line width=0.75pt}} 
   	\begin{center}
   	\begin{tikzpicture}[x=0.66pt,y=0.66pt,yscale=-0.85,xscale=0.85]
   		
   		\draw  [fill={rgb, 255:red, 0; green, 0; blue, 0}  ,fill opacity=0.25 ,even odd rule] (229.57,220.68) .. controls (229.57,176.48) and (265.4,140.64) .. (309.61,140.64) .. controls (353.81,140.64) and (389.64,176.48) .. (389.64,220.68) .. controls (389.64,264.89) and (353.81,300.72) .. (309.61,300.72) .. controls (265.4,300.72) and (229.57,264.89) .. (229.57,220.68) -- cycle ;
   		\draw  [draw opacity=0] (151.3,200.91) .. controls (160.98,122.64) and (227.43,61.78) .. (308.44,61.19) .. controls (390.52,60.59) and (458.58,122.09) .. (468.01,201.74) -- (309.61,220.68) -- cycle ; \draw   (151.3,200.91) .. controls (160.98,122.64) and (227.43,61.78) .. (308.44,61.19) .. controls (390.52,60.59) and (458.58,122.09) .. (468.01,201.74) ;  
   		\draw  [draw opacity=0] (467.97,239.99) .. controls (458.46,318.79) and (391.46,379.96) .. (310.05,380.18) .. controls (227.82,380.41) and (159.95,318.39) .. (151.07,238.49) -- (309.61,220.68) -- cycle ; \draw   (467.97,239.99) .. controls (458.46,318.79) and (391.46,379.96) .. (310.05,380.18) .. controls (227.82,380.41) and (159.95,318.39) .. (151.07,238.49) ;  
   		\draw  [draw opacity=0] (428.87,225.99) .. controls (425.88,289.48) and (373.38,340.02) .. (309.1,339.96) .. controls (284.35,339.94) and (261.37,332.43) .. (242.29,319.57) -- (309.2,220.28) -- cycle ; \draw   (428.87,225.99) .. controls (425.88,289.48) and (373.38,340.02) .. (309.1,339.96) .. controls (284.35,339.94) and (261.37,332.43) .. (242.29,319.57) ;  
   		\draw  [draw opacity=0] (269.19,107.94) .. controls (303.86,95.58) and (343.76,99.32) .. (376.62,121.52) .. controls (406.91,141.99) and (424.87,173.87) .. (428.63,207.54) -- (309.61,220.68) -- cycle ; \draw   (269.19,107.94) .. controls (303.86,95.58) and (343.76,99.32) .. (376.62,121.52) .. controls (406.91,141.99) and (424.87,173.87) .. (428.63,207.54) ;  
   		\draw  [draw opacity=0] (228.18,308.43) .. controls (211.18,292.81) and (198.51,272.11) .. (192.76,247.96) .. controls (179.75,193.19) and (206.56,138.29) .. (254.45,113.7) -- (309.2,220.28) -- cycle ; \draw   (228.18,308.43) .. controls (211.18,292.81) and (198.51,272.11) .. (192.76,247.96) .. controls (179.75,193.19) and (206.56,138.29) .. (254.45,113.7) ;  
   		\draw  [draw opacity=0] (409.4,227.21) .. controls (406.15,277.46) and (365.89,317.41) .. (315.87,320.49) -- (309.61,220.68) -- cycle ; \draw   (409.4,227.21) .. controls (406.15,277.46) and (365.89,317.41) .. (315.87,320.49) ;  
   		\draw  [draw opacity=0] (315.24,120.46) .. controls (366.06,123.49) and (406.67,164.47) .. (409.09,215.47) -- (309.2,220.28) -- cycle ; \draw   (315.24,120.46) .. controls (366.06,123.49) and (406.67,164.47) .. (409.09,215.47) ;  
   		\draw  [draw opacity=0] (209.8,214.37) .. controls (213.02,162.8) and (255.23,122.04) .. (306.99,120.72) -- (309.61,220.68) -- cycle ; \draw   (209.8,214.37) .. controls (213.02,162.8) and (255.23,122.04) .. (306.99,120.72) ;  
   		\draw  [draw opacity=0] (306.06,320.62) .. controls (253.77,318.8) and (211.73,276.85) .. (209.68,224.68) -- (309.61,220.68) -- cycle ; \draw   (306.06,320.62) .. controls (253.77,318.8) and (211.73,276.85) .. (209.68,224.68) ;  
   		\draw   (100,45) .. controls (139,2) and (426,47) .. (476,49) .. controls (526,51) and (551.41,333.12) .. (556,340) .. controls (560.59,346.88) and (421.12,281.93) .. (506.56,362.46) .. controls (592,443) and (188,410) .. (131,386) .. controls (74,362) and (61,88) .. (100,45) -- cycle ;
   		\draw [color={rgb, 255:red, 126; green, 211; blue, 33 }  ,draw opacity=1 ]   (79,117.5) .. controls (119,87.5) and (166,118.5) .. (186,120.5) ;
   		\draw [color={rgb, 255:red, 126; green, 211; blue, 33 }  ,draw opacity=1 ]   (166,150) .. controls (206,120) and (172,179.5) .. (212,149.5) ;
   		\draw [color={rgb, 255:red, 126; green, 211; blue, 33 }  ,draw opacity=1 ]   (227,132.5) .. controls (267,102.5) and (214,166.5) .. (254,136.5) ;
   		\draw [color={rgb, 255:red, 126; green, 211; blue, 33 }  ,draw opacity=1 ]   (263,144.5) -- (269,150.5) ;
   		\draw [color={rgb, 255:red, 254; green, 72; blue, 1 }  ,draw opacity=1 ]   (407,222.5) .. controls (448,204.5) and (497,239.5) .. (537,209.5) ;
   		\draw [color={rgb, 255:red, 249; green, 4; blue, 4 }  ,draw opacity=1 ]   (389.64,220.68) -- (398,224.5) ;
   		\draw [color={rgb, 255:red, 74; green, 144; blue, 226 }  ,draw opacity=1 ]   (101,344.5) .. controls (141,314.5) and (164,369.5) .. (204,339.5) ;
   		\draw [color={rgb, 255:red, 74; green, 144; blue, 226 }  ,draw opacity=1 ]   (178,311) .. controls (209,316) and (211,331.5) .. (251,301.5) ;
   		\draw [color={rgb, 255:red, 74; green, 144; blue, 226 }  ,draw opacity=1 ]   (259,295.5) -- (271,291.5) ;
   		\draw  [line width=3] [line join = round][line cap = round] (151,239) .. controls (151,239.33) and (151,239.67) .. (151,240) ;
   		\draw  [line width=3] [line join = round][line cap = round] (241,319) .. controls (241,319) and (241,319) .. (241,319) ;
   		\draw  [line width=3] [line join = round][line cap = round] (468,241) .. controls (468,241) and (468,241) .. (468,241) ;
   		\draw  [line width=3] [line join = round][line cap = round] (468,200) .. controls (468,199.67) and (468,199.33) .. (468,199) ;
   		\draw  [line width=3] [line join = round][line cap = round] (151,202) .. controls (151,202) and (151,202) .. (151,202) ;
   		\draw  [line width=3] [line join = round][line cap = round] (429,225) .. controls (429,224.67) and (429,224.33) .. (429,224) ;
   		
   		\draw (209,216) node [anchor=north west][inner sep=0.75pt]   [align=left] {...};
   		\draw (398,222) node [anchor=north west][inner sep=0.75pt]   [align=left] {..};
   		\draw (256,137) node [anchor=north west][inner sep=0.75pt]   [align=left] {.};
   		\draw (259,140) node [anchor=north west][inner sep=0.75pt]   [align=left] {.};
   		\draw (88,88) node [anchor=north west][inner sep=0.75pt]   [align=left] {$\gamma_1$};
   		\draw (498,225) node [anchor=north west][inner sep=0.75pt]   [align=left] {$\gamma_2$};
   		\draw (250,296) node [anchor=north west][inner sep=0.75pt]   [align=left] {.};
   		\draw (253,293) node [anchor=north west][inner sep=0.75pt]   [align=left] {.};
   		\draw (105,317) node [anchor=north west][inner sep=0.75pt]   [align=left] {$\gamma_3$};
   		\draw (513,388.5) node [anchor=north west][inner sep=0.75pt]   [align=left] {$J$};
   		\draw (292,211.5) node [anchor=north west][inner sep=0.75pt]  [align=left] {$p_{0,0}$};
   		\draw (411,343) node [anchor=north west][inner sep=0.75pt]   [align=left] {$p_{1,0}$};
   		\draw (199,66) node [anchor=north west][inner sep=0.75pt] [align=left] {$p_{1,1}$};
   		\draw (396,301.5) node [anchor=north west][inner sep=0.75pt]  [font=\scriptsize] [align=left] {$p_{2,0}$};
   		\draw (402,128.5) node [anchor=north west][inner sep=0.75pt]   [font=\scriptsize] [align=left] {$p_{2,1}$};
   		\draw (158,200) node [anchor=north west][inner sep=0.75pt]  [font=\scriptsize]   [align=left] {$p_{2,2}$};
   		\draw (369,257) node [anchor=north west][inner sep=0.75pt]  [font=\scriptsize] [align=left] {$p_{3,0}$};
   		\draw (341,138) node [anchor=north west][inner sep=0.75pt]  [font=\scriptsize]  [align=left] {$p_{3,1}$};
   		\draw (231.5,152.5) node [anchor=north west][inner sep=0.75pt]  [font=\scriptsize] [align=left] {$p_{3,2}$};
   		\draw (231.5,274.5) node [anchor=north west][inner sep=0.75pt]  [font=\scriptsize]  [align=left] {$p_{3,3}$};
   		\draw (114.5,228) node [anchor=north west][inner sep=0.75pt]   [align=left] {$a_{1,0}$};
   		\draw (118,194) node [anchor=north west][inner sep=0.75pt]   [align=left] {$b_{1,1}$};
   		\draw (472,190) node [anchor=north west][inner sep=0.75pt]   [align=left] {$a_{1,1}$};
   		\draw (472,236) node [anchor=north west][inner sep=0.75pt]   [align=left] {$b_{1,0}$};
   		\draw (223,328) node [anchor=north west][inner sep=0.75pt]   [align=left] {$a_{2,0}$};
   		\draw (431,220) node [anchor=north west][inner sep=0.75pt]   [align=left] {$b_{2,0}$};
   		
   	\end{tikzpicture}
   	\end{center}
   	\caption{Some complementary components of $U_\alpha$ and the Jordan curve $J \subset U_\alpha$.}
   	\label{ua}
   \end{figure}

    By the above construction, we have 
    $$\begin{aligned} \sum_{p \in C(U_\alpha)} \operatorname{diam}(p)^\alpha & \leq 2^\alpha+\sum_{k=1}^{\infty} \diam \left(p_{k, j}\right)^\alpha \cdot (k+1) \\ & <2^\alpha+\sum_{k=1}^{\infty}\left(\frac{2 \pi\left(1+\frac{1}{2^{k-1}}\right)}{k+1}\right)^\alpha \cdot (k+1)  \\ & <2^\alpha+(2 \pi)^\alpha \cdot 2^{\alpha} \left(\sum_{k=1}^{\infty} \frac{1}{(k+1)^{\alpha-1}}\right)  < \infty.
    \end{aligned}$$
    The last inequality holds since $\alpha>2$. Thus $U_\alpha$ is an $\ell^\alpha$-domain. 
    
    Since $U_\alpha$ is a countably connected domain, by the He-Schramm theorem there exists a conformal homeomorphism $f: U_\alpha \to D$ 
    onto a circle domain $D$. Moreover, $f$ is unique up to postcomposition by a M\"obius transformation. Without loss of generality, we assume $\hat{f}(\infty)=\infty$. To show that $\hat{f}(\overline{\mathbb{D}}) \in \mathcal{C}_P(D)$, we fix a Jordan curve $J \subset U_\alpha$ such that if $V$ is the bounded component of $\mathbb{S}^2\setminus J$, then $\mathbb{S}^2\setminus U_\alpha \subset V$. We denote 
    $$\Gamma_\alpha:=\{\text{curves in } \hat{U}_\alpha \text{ that join } \pi_{U_\alpha}(\overline{\mathbb{D}}) \text{ and } \pi_{U_\alpha}(J) \}.$$
    
    Towards a contradiction, we assume $\hat{f}(\overline{\mathbb{D}}) \in \mathcal{C}_N(D)$. The proof can be divided into three steps. First, we define an admissible function $\rho$ for $\Gamma_\alpha$. Next, we employ $\rho$ to estimate $\Mod_\mathrm{T}(\Gamma_\alpha)$. Finally, with the help of the proof of Theorem \ref{bz} and Remark \ref{diam2yx}, we derive that such an estimate is impossible, thereby $\hat{f}(\overline{\mathbb{D}}) \in \mathcal{C}_P(D)$.
    
    We let $\rho: \hat{U}_\alpha \to [0,\infty]$, 
    \begin{eqnarray*}
    	\rho(x)=\left\{  \begin{array}{ll}
    		\frac{1}{(k+1) \log (k+1)}, & x \in p_{k,j} \,(k \geq 1), \\
    		\frac{1}{|x-c_{k,j}| \log \frac{R_k}{g_k}},  & x \in U_\alpha \cap A_{k,j} \,(k \geq 1),  \\ 
    		0, & \text{otherwise}.  
    	\end{array}
    	\right. 
    \end{eqnarray*} 
   Here  $$A_{k,j}:=B(c_{k, j},R_k)\setminus \overline{B}(c_{k, j},g_k), \quad k=1,2,\ldots \text{ and }  j=0,1,\ldots,k, $$ 
   and the center and radii are defined as follows: 
    \begin{eqnarray*}
    	c_{k, j}:=\left\{\begin{array}{l} \frac{b_{k,j}+a_{k,j+1}}{2}, \text { when } j=0, \ldots, k-1, \\ \frac{b_{k, k}+a_{k, 0}}{2}, \text { when } j=k, \end{array}\right.
    \end{eqnarray*}
  \begin{equation}\label{gk}
    	g_k:=\left|\frac{a_{k, j+1}-b_{k, j}}{2}\right|<\theta_k \cdot r_k, \quad \text{and} \quad R_k:= r_k\cdot {\theta_k}e^{2^k}>e^{2^k} g_k.
    \end{equation}
   
   Recall that since $r_k=1+\frac{1}{2^{k-1}}$,
    the first inequality of \eqref{gk} follows from the fact that the shortest path between two points is a line segment. Moreover, 
    since $\theta_k$ is given by \eqref{thetak}, we have
     $$R_k=\frac{2\pi r_k}{2025^k}<\frac{1}{100 \cdot 2^{k-1}},$$
    which implies that the annuli $A_{k,j}$ are pairwise disjoint. Figure \ref{sh} shows the annuli $A_{1,0}$ and $A_{2,0}$.

    \begin{figure}

 \tikzset{every picture/.style={line width=0.75pt}} 
 \begin{center}
 \begin{tikzpicture}[x=0.75pt,y=0.75pt,yscale=-0.85,xscale=0.85]
 	
 	\draw  [draw opacity=0] (356.95,138.33) .. controls (391.44,151.29) and (417.26,181.09) .. (425.35,217.15) -- (318.5,241.5) -- cycle ; \draw   (356.95,138.33) .. controls (391.44,151.29) and (417.26,181.09) .. (425.35,217.15) ;  
 	\draw  [draw opacity=0] (425.35,264.62) .. controls (416.56,303.97) and (388.09,334.16) .. (352.58,346.23) -- (318.51,240.5) -- cycle ; \draw   (425.35,264.62) .. controls (416.56,303.97) and (388.09,334.16) .. (352.58,346.23) ;  
 	\draw  [draw opacity=0] (539.57,281.64) .. controls (522.8,374.07) and (449.14,446.76) .. (356.06,462.38) -- (318.5,241.5) -- cycle ; \draw   (539.57,281.64) .. controls (522.8,374.07) and (449.14,446.76) .. (356.06,462.38) ;  
 	\draw  [draw opacity=0] (352.09,18.99) .. controls (447.65,33.26) and (523.53,107.55) .. (539.9,202.14) -- (318.51,240.5) -- cycle ; \draw   (352.09,18.99) .. controls (447.65,33.26) and (523.53,107.55) .. (539.9,202.14) ;  
 	\draw  [line width=3] [line join = round][line cap = round] (425,216) .. controls (425,216) and (425,216) .. (425,216) ;
 	\draw  [line width=3] [line join = round][line cap = round] (539,204) .. controls (539,204) and (539,204) .. (539,204) ;
 	\draw  [line width=3] [line join = round][line cap = round] (539,280) .. controls (539,280) and (539,280) .. (539,280) ;
 	\draw  [line width=3] [line join = round][line cap = round] (426,265) .. controls (426,265) and (426,265) .. (426,265) ;
 	\draw  [dash pattern={on 0.84pt off 2.51pt}]  (539.88,203.14) -- (539.57,281.64) ;
 	\draw  [dash pattern={on 0.84pt off 2.51pt}]  (425.35,217.15) -- (425.35,264.62) ;
 	\draw  [fill={rgb, 255:red, 0; green, 0; blue, 0 }  ,fill opacity=0.25 ,even odd rule] (500.42,242.39) .. controls (500.42,221.84) and (518.02,205.18) .. (539.73,205.18) .. controls (561.44,205.18) and (579.04,221.84) .. (579.04,242.39) .. controls (579.04,262.94) and (561.44,279.6) .. (539.73,279.6) .. controls (518.02,279.6) and (500.42,262.94) .. (500.42,242.39)(474.79,242.39) .. controls (474.79,207.68) and (503.86,179.55) .. (539.73,179.55) .. controls (575.59,179.55) and (604.66,207.68) .. (604.66,242.39) .. controls (604.66,277.1) and (575.59,305.23) .. (539.73,305.23) .. controls (503.86,305.23) and (474.79,277.1) .. (474.79,242.39) ;
 	\draw  [fill={rgb, 255:red, 0; green, 0; blue, 0}  ,fill opacity=0.25 ,even odd rule] (402,240.88) .. controls (402,227.99) and (412.45,217.53) .. (425.35,217.53) .. controls (438.24,217.53) and (448.69,227.99) .. (448.69,240.88) .. controls (448.69,253.77) and (438.24,264.23) .. (425.35,264.23) .. controls (412.45,264.23) and (402,253.77) .. (402,240.88)(384.27,240.88) .. controls (384.27,218.19) and (402.66,199.8) .. (425.35,199.8) .. controls (448.03,199.8) and (466.42,218.19) .. (466.42,240.88) .. controls (466.42,263.57) and (448.03,281.96) .. (425.35,281.96) .. controls (402.66,281.96) and (384.27,263.57) .. (384.27,240.88) ;
 	\draw  [line width=3] [line join = round][line cap = round] (426,241) .. controls (426,241) and (426,241) .. (426,241) ;
 	\draw  [line width=3] [line join = round][line cap = round] (540,240) .. controls (540,240) and (540,240) .. (540,240) ;
 	\draw    (331,224.5) .. controls (346,214.5) and (623,271.5) .. (630,257.5) ;
 	
 	\draw (486,299) node [anchor=north west][inner sep=0.75pt]  [font=\small] [align=left] {$A_{1,0}$};
 	\draw (365,263) node [anchor=north west][inner sep=0.75pt]  [font=\small] [align=left] {$A_{2,0}$};
 	\draw (519,190) node [anchor=north west][inner sep=0.75pt]  [font=\small] [align=left] {$a_{1,1}$};
 	\draw (522,277) node [anchor=north west][inner sep=0.75pt]  [font=\small] [align=left] {$b_{1,0}$};
 	\draw (307,226) node [anchor=north west][inner sep=0.75pt]   [align=left] {....};
 	\draw (405,262) node [anchor=north west][inner sep=0.75pt]  [font=\footnotesize] [align=left] {$b_{2,0}$};
 	\draw (405,204) node [anchor=north west][inner sep=0.75pt]  [font=\footnotesize] [align=left] {$a_{2,1}$};
 	\draw (515,225) node [anchor=north west][inner sep=0.75pt]  [font=\small] [align=left] {$c_{1,0}$};
 	\draw (410,242) node [anchor=north west][inner sep=0.75pt]  [font=\footnotesize] [align=left] {$c_{2,0}$};
 	\draw (619,265) node [anchor=north west][inner sep=0.75pt]  [font=\small] [align=left] {$\gamma$};
 	\draw (488,75) node [anchor=north west][inner sep=0.75pt]   [align=left] {$p_{1,1}$};
 	\draw (486,391) node [anchor=north west][inner sep=0.75pt]   [align=left] {$p_{1,0}$};
 	\draw (399,313) node [anchor=north west][inner sep=0.75pt]   [align=left] {$p_{2,0}$};
 	\draw (391,142) node [anchor=north west][inner sep=0.75pt]   [align=left] {$p_{2,1}$};

 \end{tikzpicture}
    \end{center}
\caption{The annuli $A_{1,0}$ and $A_{2,0}$ in the definition of $\rho$. }
\label{sh}
    \end{figure}
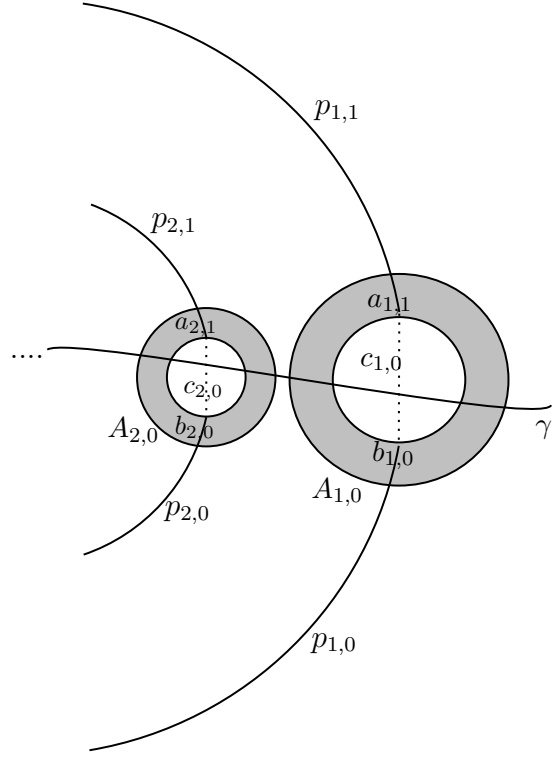

    \begin{lemma}\label{adm}
    	$\rho$ is admissible for $\Gamma_\alpha$ and 
    	\begin{eqnarray*}
    		\int_{\gamma} \rho \, ds +\sum_{p \in \mathcal{C}(U_\alpha) \cap |\gamma|} \rho(p)=\infty  \quad \text{for all } \gamma \in \Gamma_\alpha.  
    	\end{eqnarray*}
    \end{lemma}
 \begin{proof}
    Fix a curve $\gamma \in \Gamma_\alpha$. We consider cases based on the number of times $\gamma$ passes through complementary components or "gaps" between them; see Figure \ref{ua} 
    for illustration. 
    
    \textbf{Case 1:}
    Suppose that $|\gamma| \cap \overline{B}(c_{k, j},g_k) \neq \emptyset$ only for finitely many $k=1,2,\ldots$. Then there exists a $K_1>0$ such that for every $k\geq K_1$, we have $|\gamma|\cap p_{k,j} \neq \emptyset$ for some $j(k) \in \{0, \ldots, k\}$. This implies that 
    \[
    \int_{\gamma} \rho \, ds +\sum_{p \in \mathcal{C}(U_\alpha) \cap |\gamma|} \rho(p)\geq \sum_{k=K_1}^{\infty} \frac{1}{(k+1) \log (k+1)}=\infty. 
    \]
    
     \textbf{Case 2:} Suppose that there is an infinite set $K \subset \{1,2,\ldots\}$ such that for every $k \in K$ there is an index $j=j(k)$ for which $|\gamma| \cap \overline{B}(c_{k, j},g_k) \neq \emptyset$. Then 
          \[
    \int_{\gamma} \rho \, ds +\sum_{p \in \mathcal{C}(U_\alpha) \cap |\gamma|} \rho(p)\geq \sum_{k\in K} \int_{|\gamma| \cap A_{k, j}} \rho \, ds \geq \sum_{k\in K} \int_{g_{k}}^{R_{k}} \frac{dr}{r \log {\frac{R_{k}}{g_{k}}}}=\infty. 
    \]
\end{proof}
 It follows from Lemma \ref{adm} that for every $\epsilon>0$ the function $\widetilde{\rho}:= \epsilon\rho$ is also admissible for $\Gamma_\alpha$. Therefore, 
        \[
    	\begin{aligned} \Mod_\mathrm{T} (\Gamma_\alpha) & \leq \int_{U_\alpha} \epsilon^2 \rho^2 \, d\mathcal{H}^2+\sum_{p \in \mathcal{C}(U_\alpha)} \epsilon^2\rho(p)^2 \  \\ & =\sum_{k=1}^{\infty} \int_{A_{k, j}} \frac{\epsilon^2 \, d\mathcal{H}^2}{|x-c_{k,j}|^2 \log^2 \frac{R_k}{g_k}}+\sum_{k=1}^{\infty} \frac{\epsilon^2(1+k)}{(1+k)^2 \log ^2(1+k)}\\ &  \leq \epsilon^2\left(\sum_{k=1}^{\infty} \int_{g_k}^{R_k} \frac{2 \pi}{r \log ^2 \frac{R_k}{g_k}} d r+\sum_{k=1}^{\infty} \frac{1}{(1+k) \log ^2(1+k)}\right) \\ & \leq \epsilon^2\left(2 \pi \sum_{k=1}^{\infty} \frac{1}{2^k}+\sum_{k=1}^{\infty} \frac{1}{(1+k) \log ^2(1+k)} \right)\\& \leq \epsilon^2 (4\pi+M).\end{aligned}
   \]
   Here $M<\infty$ denotes the sum of the second series, and the third inequality follows from \eqref{gk}. Since the above estimate holds for all $\epsilon>0$, by letting $\epsilon\to 0$ we deduce that 
    \begin{equation}\label{gammmaa}
    	\Mod_\mathrm{T} (\Gamma_\alpha)=0.
    \end{equation}
    
    On the other hand, if $\hat{f}(\overline{\mathbb{D}}) \in \mathcal{C}_N(D)$, then using the same argument as in Lemma \ref{pm} and Remark \ref{diam2yx}, we have
    \begin{eqnarray*}
    	\Mod_\mathrm{T} (\hat{f}(\Gamma_\alpha))>0.
    \end{eqnarray*}
    This contradicts \eqref{gammmaa} and the conformal invariance of the transboundary modulus (Theorem \ref{qcinvariance}), so we have $\hat{f}(\overline{\mathbb{D}}) \in \mathcal{C}_P(D)$ for every conformal homeomorphism $f: U_\alpha \to D$ onto a circle domain $D$. 
\end{proof}


\section{Proof of Theorem \ref{ptonon2}}
In this section we show that no $\ell^\alpha$-assumption on the complementary components is sufficient to guarantee the last conclusion of Theorem \ref{main}, i.e., that point-components are mapped to point-components. Recall that $\mathcal{F}_\alpha(X)$ is the collection of countably connected subdomains $\Omega_\alpha$ of $X$ satisfying  
	$\sum_{p \in \mathcal{C}(\Omega_\alpha)} \operatorname{diam}(p)^\alpha<\infty$.

	\begin{theoremx}[Theorem \ref{ptonon2}]
	For every $\alpha>0$ there is a domain $\Omega_\alpha \in \mathcal{F}_\alpha(\mathbb{S}^2)$ such that $\{0\} \in \mathcal{C}_P(\Omega_\alpha)$ but $\hat{f}(\{0\}) \in \mathcal{C}_N(D)$ for every conformal homeomorphism $f: \Omega_\alpha \to D$ onto a circle domain $D \subset \mathbb{S}^2$.
   \end{theoremx}
   Recall that such a conformal homeomorphism $f$ always exists by the He-Schramm theorem. 

\begin{proof}[Proof of Theorem \ref{ptonon2}]
	Fix $\alpha>0$. We construct a domain $\Omega_\alpha$ by describing each element of $\mathcal{C}(\Omega_\alpha)$. The origin $\{0\}$ is the only element of $\mathcal{C}_P(\Omega_\alpha)$. We parametrize the non-trivial complementary components as follows: Given $j=1,2,\ldots$, we denote by $W_j$ the collection of finite words $w=w_1\cdots w_j$, where $w_k \in \{0,1\}$ for every $1 \leq k \leq j$. 
	Moreover, denote $W_0:=\{\emptyset\}$ and $W:=\bigcup_{j=0}^\infty W_j$. Then
	$$
	\mathcal{C}_N(\Omega_\alpha)=\{p_w: \, w \in W\}.   
	$$
			
     \begin{figure}
     \tikzset{every picture/.style={line width=0.75pt}} 
	\begin{center}  
     \begin{tikzpicture}[x=0.75pt,y=0.75pt,yscale=-0.85,xscale=0.85]
	
	\draw  [line width=3] [line join = round][line cap = round] (329.5,341) .. controls (329.5,341) and (329.5,341) .. (329.5,341) ;
	\draw    (376,340) -- (597,338) ;
	\draw    (492,201) -- (360,312) ;
	\draw    (356,370) -- (480,488) ;
	\draw    (419,197) -- (340,320) ;
	\draw    (481,268) -- (351,329) ;
	\draw    (482,409) -- (351,350) ;
	\draw  [dash pattern={on 0.84pt off 2.51pt}] (161.19,339) .. controls (161.19,246.32) and (236.32,171.19) .. (329,171.19) .. controls (421.68,171.19) and (496.81,246.32) .. (496.81,339) .. controls (496.81,431.68) and (421.68,506.81) .. (329,506.81) .. controls (236.32,506.81) and (161.19,431.68) .. (161.19,339) -- cycle ;
	\draw    (407,488.25) -- (333,355.25) ;
	\draw    (364,249) -- (334,324) ;
	\draw    (391,265) -- (340,327) ;
	\draw    (410,287) -- (342,331) ;
	\draw    (422,314) -- (345.25,337) ;
	\draw    (423,361) -- (345,343) ;
	\draw    (408,394) -- (341,351) ;
	\draw    (385,417) -- (337,354) ;
	\draw    (352,433) -- (331,359.25) ;
	\draw    (341,274) -- (330,332) ;
	\draw    (328,355.25) -- (335,405) ;
	\draw  [dash pattern={on 0.84pt off 2.51pt}] (232.94,339) .. controls (232.94,285.95) and (275.95,242.94) .. (329,242.94) .. controls (382.05,242.94) and (425.06,285.95) .. (425.06,339) .. controls (425.06,392.05) and (382.05,435.06) .. (329,435.06) .. controls (275.95,435.06) and (232.94,392.05) .. (232.94,339) -- cycle ;
	\draw  [dash pattern={on 0.84pt off 2.51pt}] (262.75,339) .. controls (262.75,302.41) and (292.41,272.75) .. (329,272.75) .. controls (365.59,272.75) and (395.25,302.41) .. (395.25,339) .. controls (395.25,375.59) and (365.59,405.25) .. (329,405.25) .. controls (292.41,405.25) and (262.75,375.59) .. (262.75,339) -- cycle ;
	\draw  [dash pattern={on 0.84pt off 2.51pt}] (116.28,339) .. controls (116.28,221.52) and (211.52,126.28) .. (329,126.28) .. controls (446.48,126.28) and (541.72,221.52) .. (541.72,339) .. controls (541.72,456.48) and (446.48,551.72) .. (329,551.72) .. controls (211.52,551.72) and (116.28,456.48) .. (116.28,339) -- cycle ;
	
	\draw (600,327) node [anchor=north west][inner sep=0.75pt]   [align=left] {$p_{\emptyset}$};
	\draw (495,186) node [anchor=north west][inner sep=0.75pt]   [align=left] {$p_{0}$};
	\draw (478,486) node [anchor=north west][inner sep=0.75pt]   [align=left] {$p_{1}$};
	\draw (402,172) node [anchor=north west][inner sep=0.75pt]   [align=left] {$p_{00}$};
	\draw (482,255) node [anchor=north west][inner sep=0.75pt]   [align=left] {$p_{01}$};
	\draw (342,230) node [anchor=north west][inner sep=0.75pt]   [align=left] {$p_{000}$};
	\draw (384,246) node [anchor=north west][inner sep=0.75pt]   [align=left] {$p_{001}$};
	\draw (410,269) node [anchor=north west][inner sep=0.75pt]   [align=left] {$p_{010}$};
	\draw (424,305) node [anchor=north west][inner sep=0.75pt]   [align=left] {$p_{011}$};
	\draw (485,403) node [anchor=north west][inner sep=0.75pt]   [align=left] {$p_{10}$};
	\draw (399,490) node [anchor=north west][inner sep=0.75pt]   [align=left] {$p_{11}$};
	\draw (427,351) node [anchor=north west][inner sep=0.75pt]   [align=left] {$p_{100}$};
	\draw (408,392) node [anchor=north west][inner sep=0.75pt]   [align=left] {$p_{101}$};
	\draw (376,420) node [anchor=north west][inner sep=0.75pt]   [align=left] {$p_{110}$};
	\draw (338,435) node [anchor=north west][inner sep=0.75pt]   [align=left] {$p_{111}$};
	\draw (318,260) node [anchor=north west][inner sep=0.75pt]   [align=left] {$p_{0\cdots0}$};
	\draw (307,404) node [anchor=north west][inner sep=0.75pt]   [align=left] {$p_{1\cdots1}$};
	\draw (315,334) node [anchor=north west][inner sep=0.75pt]   [align=left] {0};
	\draw (336,291) node [anchor=north west][inner sep=0.75pt]   [align=left] {.};
	\draw (340,292) node [anchor=north west][inner sep=0.75pt]   [align=left] {.};
	\draw (332,393) node [anchor=north west][inner sep=0.75pt]   [align=left] {.};
	\draw (335,392) node [anchor=north west][inner sep=0.75pt]   [align=left] {.};
	\draw (409,341) node [anchor=north west][inner sep=0.75pt]   [align=left] {.};
	\draw (408,345.5) node [anchor=north west][inner sep=0.75pt]   [align=left] {.};
	\draw (407,350) node [anchor=north west][inner sep=0.75pt]   [align=left] {.};
	\draw (409,333) node [anchor=north west][inner sep=0.75pt]   [align=left] {.};
	\draw (408,328.5) node [anchor=north west][inner sep=0.75pt]   [align=left] {.};
	\draw (407,324) node [anchor=north west][inner sep=0.75pt]   [align=left] {.};

    \end{tikzpicture}
	\end{center}
	\caption{Some complementary components of $\Omega_\alpha$.}
	\label{oa}
    \end{figure}
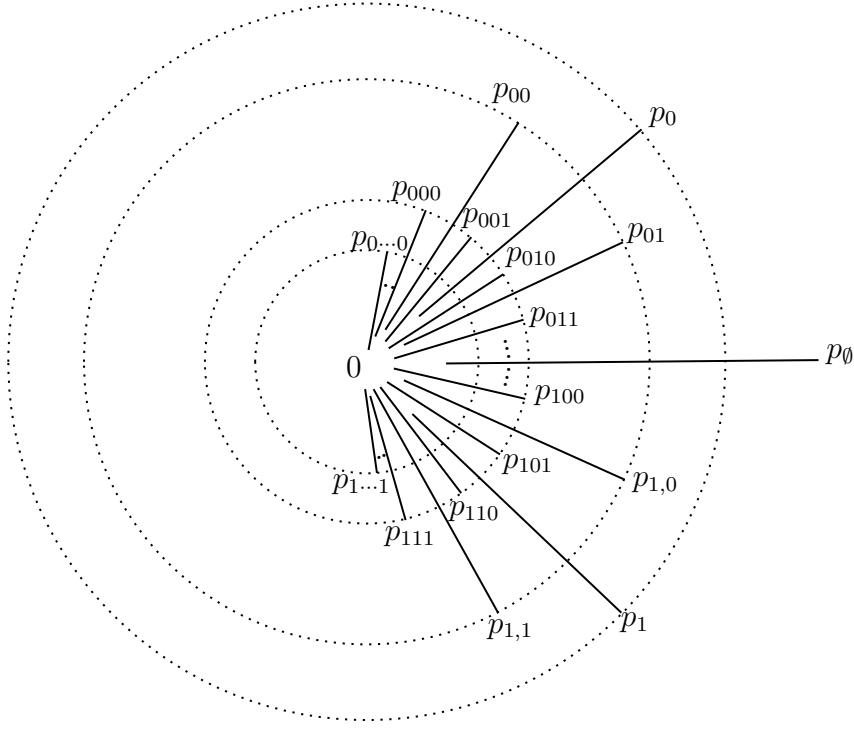
	
	We now define the elements $p_w$. For each $n=0,1,2,\ldots$, consider the sequences  
	\begin{equation}\label{ew2}
		\ell_{n}:=\Big(\frac{1}{2^{n+n_0}\cdot (n+n_0)^2}\Big)^\frac{1}{\alpha}>0 \quad \text{and} \quad \epsilon_{n}:=\frac{\ell_{n+1}}{2^{n+n_0+1}}>0, 
	\end{equation} 
	where $n_0 = n_0(\alpha) \geq 1$ is a large enough integer such that 
	\begin{equation}\label{n0}
		\epsilon_{n-1} \leq \epsilon_{n+1}+\ell_{n+1} \quad \text{for every } n =1,2,\ldots ; 
	\end{equation} 
	an elementary computation shows that such an $n_0$ exists. 
	We apply an inductive definition. First, let 
	$$
	p_\emptyset=[\epsilon_0,\epsilon_0+\ell_0] 
	$$
	be a segment on the positive $x$-axis. Next, fix $j=1,2,\ldots$ and suppose that $w=\bar w w_j$, where $\bar w \in W_{j-1}$ and $w_j \in \{0,1\}$. Moreover, assume that 
	\begin{equation} \label{porkki}
	p_{\bar w}:=e^{i\theta_{\bar w}}\cdot \left[\epsilon_{j-1}, \epsilon_{j-1}+\ell_{j-1}\right]
	\end{equation}   
	is defined in polar coordinates. We set  
	\begin{equation} \label{norkki}
	p_w:=e^{i \theta_w}\cdot \left[\epsilon_{j}, \epsilon_{j}+\ell_{j}\right], 
	\end{equation} 
	where the angle $\theta_w$ is given by
	\begin{equation}\label{theta}
		\theta_w:= \begin{cases}\theta_{\bar{w}}+\alpha_{j}, & \text { if } w_j=0, \\ \theta_{\bar{w}}-\alpha_{j}, & \text { if } w_j=1, \end{cases} \quad \quad \alpha_j:=4^{-10j} \log \Big(1+\frac{1}{2^{j+n_0+3}}\Big).
	\end{equation} 
	Notice that the sequence $(\alpha_j)_j$ decreases rapidly towards $0$ as $j \to \infty$. Figure \ref{oa} depicts the first few generations of the complementary components $p_w$. We will apply the following properties. 
	
	\begin{lemma} \label{terveystalo}
	Let $w=\bar w w_j$ be as above. The following properties hold: 
	\begin{enumerate} 
	\item \label{one} The elements $p_w$ are pairwise disjoint subsets of the right half-plane
	\item \label{two} For every $\epsilon>0$ there is a $j_\epsilon \geq 1$ such that if $j \geq j_\epsilon$, then $p_w \subset B(0,\epsilon)$. 
	\item \label{three} There is a family $\Gamma_w$ of curves joining $p_w$ and $p_{\bar{w}}$ in $\Omega_\alpha$, 
	such that $\Mod_\mathrm{C}(\Gamma_w) > 4^{j}$. 
	\item \label{four} We have $\sum_{p \in \mathcal{C}(\Omega_\alpha)} \operatorname{diam}(p)^{\alpha}<\infty$; thus $\Omega_\alpha \in \mathcal{F}_\alpha(\mathbb{S}^2)$. 
	\end{enumerate} 
	\end{lemma} 
	
	\begin{proof}
	Since $p_\emptyset$ lies on the positive real axis, Condition \eqref{one} follows from the construction \eqref{theta} of the angles $\theta_w$ and the choice of the sequence $(\alpha_j)_j$.   
	To prove Condition \eqref{two}, we notice that by \eqref{ew2}, both $\ell_j$ and $\epsilon_j$ 
	converge to $0$ as $j \to \infty$. Consequently, we can choose a $j_\epsilon \geq 1$ such that $\ell_j+\epsilon_j<\epsilon$ for every $j \geq j_\epsilon$, which implies 
	that $p_w \subset B(0,\epsilon)$ for all $w \in W_j$, $j \geq j_\epsilon$.   
	
	We now prove Condition \eqref{three}. We claim that for every 
	\begin{equation} 
	\label{vimppa} \epsilon_{j+1}+\ell_{j+1} < t < \epsilon_j+\ell_j 
	\end{equation} 
	there is a subarc $J_t$ of $S(0,t)$ so that 
	\begin{itemize} 
	\item[(i)] the length of $J_t$ is $t \alpha_j$,   
	\item[(ii)] one endpoint of $J_t$ lies in $p_{\bar w}$ and the other endpoint lies in $p_w$,  
	\item[(iii)] all other points of $J_t$ lie in $\Omega_\alpha$. 
	\end{itemize} 
	Indeed, if $\epsilon_{j-1} < t < \epsilon_j+\ell_j$, then $S(0,t)$ intersects both $p_{\bar w}$ and $p_w$ by \eqref{porkki} and \eqref{norkki}. We let $J_t$ be the shorter of the two 
	subarcs of $S(0,t)$ 		joining $p_{\bar w}$ and $p_w$. Then $J_t$ satisfies Properties (i) and (ii). By \eqref{n0}, such properties remain true for all $t$ that satisfy \eqref{vimppa}. 
	
	To prove Property (iii), we fix a $w' \in W_{j'}$ other than $\bar w$ or $w$. The choice of the sequence $(\alpha_j)_j$ in \eqref{theta} then guarantees that if $t$ satisfies 
	\eqref{vimppa} then $p_{w'} \cap J_t = \emptyset$ 
	if $j' \leq j$. Finally, the lower bound in \eqref{vimppa} guarantees that $p_{w'} \cap J_t = \emptyset$ also when $j' > j$. We conclude that Property (iii) holds.    
	
	We set $\Gamma_w=\{J_t: \, \epsilon_{j+1}+\ell_{j+1}  < t < \epsilon_j+\ell_j \}$, and recall from the definitions of $\ell_j$ and $\epsilon_j$ in \eqref{ew2} that 
	\begin{equation} 
	\label{kikka}
	\frac{\epsilon_j+\ell_j}{\epsilon_{j+1}+\ell_{j+1}} \geq \frac{\epsilon_j+\ell_{j+1}}{(1+\frac{1}{2^{j+n_0+2}})\ell_{j+1}} = 1+\frac{1}{2^{j+n_0+2}+1}. 
	\end{equation} 
	
	Let $\rho$ be admissible for $\Gamma_w$. Then $1 \leq \int_{J_t} \rho \, ds$ for every $J_t \in \Gamma_w$. We divide both sides by $t$ and integrate over $t$ to get 
	\begin{eqnarray} \label{nakki1}
	\log \Big(\frac{\epsilon_j + \ell_j}{\epsilon_{j+1}+\ell_{j+1}} \Big) \leq \int_{\epsilon_{j+1}+\ell_{j+1}}^{\epsilon_j + \ell_j} \frac{dt}{t}  \leq \int_{A_w} \frac{\rho(x)}{|x|} \, d\mathcal{H}^2(x)=:I, 
	\end{eqnarray} 
	where $A_w=\{x \in \Omega_\alpha: \, x \in J_t \text{ for some }J_t \in \Gamma_w\}$. We apply H\"older's inequality and polar coordinates to see that 
	\begin{eqnarray} \nonumber 
	I^2  &\leq&   \Big( \int_{A_w} |x|^{-2} \, d\mathcal{H}^2(x) \Big) \Big(\int_{\Omega_\alpha} \rho(x)^2 \, d\mathcal{H}^2(x)  \Big) \\ \label{nakki2} &=&  
	\alpha_j  \log \Big(\frac{\epsilon_j + \ell_j}{\epsilon_{j+1}+\ell_{j+1}} \Big) \Big(\int_{\Omega_\alpha} \rho(x)^2 \, d\mathcal{H}^2(x)  \Big). 
	\end{eqnarray}
	Combining \eqref{kikka}, \eqref{nakki1}, and \eqref{nakki2} with \eqref{theta}, and minimizing over all admissible functions $\rho$, then yields the desired bound $\Mod_C \Gamma_w >4^{j}$. 
	
	Finally, since the set $W_j$ has $2^j$ elements for every $j=0,1,2,\ldots$,  
	$$ 
	\sum_{p \in \mathcal{C}(\Omega_\alpha)} \operatorname{diam}(p)^{\alpha}=\sum_{j=0}^{\infty} 2^j \ell_j^\alpha \leq \sum_{j=0}^\infty (j+n_0)^{-2} < \infty 
	$$ 
	by the choice \eqref{ew2} of the sequence $(\ell_j)_j$. Condition \eqref{four} follows. 
	\end{proof}

	Since $\Omega_\alpha$ is a countably connected domain, by the He-Schramm theorem there exists a conformal homeomorphism $f: \Omega_\alpha \to D$ 
	onto a circle domain $D$. Furthermore, $f$ is unique up to postcomposition by a M\"obius transformation. To show that $\hat{f}(\{0\}) \in \mathcal{C}_N(D)$, 
	we assume towards a contradiction that $\hat{f}(\{0\})$ is a point-component. 
	
	We denote by $\Gamma$ the family of curves in $\hat{\Omega}_\alpha$ joining $p_\emptyset$ and $\{0\}$. As in Remark \ref{diam2yx}, we can adapt the proof of Proposition 
	\ref{asyb} to show that $\Mod_\mathrm{T}(\hat{f}(\Gamma)) =0$.
	By the conformal invariance of the transboundary modulus (Theorem \ref{qcinvariance}), we have $\Mod_\mathrm{T}(\Gamma)=0$. The desired contradiction follows if we can prove that 
	\begin{equation}
		\label{zmo}
		\Mod_\mathrm{T}(\Gamma)>0. 
	\end{equation} 
	
	We denote by $W_\infty$ the collection of all the infinite words $w_1w_2\cdots$, where $w_k \in \{0,1\}$. We equip the space $W_\infty$ with the unique probability measure $\nu$ 
	satisfying $\nu(U_{w})=2^{-j}$ for all $j \geq 1$ and $w \in W_j$, where 
	$$
	U_w:=\{w_\infty \in W_\infty: \, w_\infty=ww_{j+1}w_{j+2}\cdots \}. 
	$$
	
	Let $\rho:\hat{\Omega}_\alpha \to [0,\infty]$ be an admissible function for $\Gamma$ that satisfies  
	\begin{equation} 
		\label{rohma}
		\int_{\Omega_\alpha} \rho^2\, d\mathcal{H}^2 + \sum_{w \in W} \rho(w)^2=1. 
	\end{equation} 
	We claim that there exists a $u_\infty=u_1u_2\cdots \in W_\infty$ so that 
	\begin{equation} \label{xy1} 
		\sum_{j=1}^\infty \rho(p_{\widetilde{u}_j}) \leq 1.  
	\end{equation} 
	Here $\widetilde{u}_j:=u_{1}u_2\cdots u_j$. Indeed, by Fubini's theorem we have 
	\begin{eqnarray*} 
		\int_{W_\infty} \sum_{j=1}^\infty \rho(p_{\widetilde{u}_j}) \, d\nu(u_\infty) 
		= \sum_{j=1}^\infty \sum_{w \in W_j} \nu(U_w)\rho(p_w)
		= \sum_{j=1}^\infty 2^{-j} \sum_{w \in W_j} \rho(p_w)=:I_\infty. 
	\end{eqnarray*} 
	By the Cauchy-Schwarz inequality and since $\#W_j=2^{j}$,
	\begin{eqnarray*}
		I_\infty \leq \sum_{j=1}^\infty 2^{-j/2} \Big(\sum_{w \in W_j} \rho(p_w)^2 \Big)^{1/2} 
		\leq \Big( \sum_{j=1}^\infty 2^{-j} \Big)^{1/2} \Big( \sum_{w \in W} \rho(p_w)^2 \Big)^{1/2} \leq1, 
	\end{eqnarray*} 
	where the last inequality follows from \eqref{rohma}. Combining the estimates shows that there exists a $u_\infty=u_1u_2\cdots \in W_\infty$ satisfying \eqref{xy1}. 
	
	By Lemma \ref{terveystalo}, for every $\widetilde{u}_j=u_{1}u_2\cdots u_j$, $j=1,2,\ldots$, there is a curve family $\Gamma_{\widetilde{u}_j}$ joining $p_{\widetilde{u}_{j-1}}$ and $p_{\widetilde{u}_j}$ in $\Omega_\alpha$ such that 
	$\Mod_\mathrm{C}(\Gamma_{\widetilde{u}_j}) > 4^{j}$. We next claim that for every such $j$ there is a $\gamma_j \in \Gamma_{\widetilde{u}_j}$ such that 
	\begin{equation} \label{ubles}
		\int_{\gamma_j} \rho \, ds \leq 2^{-j}. 
	\end{equation}
	Assume towards a contradiction that there is a $j_0\geq 1$ such that $\int_{\gamma_{j_0}} \rho \, ds > 2^{-j_0}$ for every $\gamma_{j_0} \in \Gamma_{\widetilde{u}_{j_0}} $. Then $2^{j_0}\rho$ is an admissible function for $\Gamma_{\widetilde{u}_{j_0}}$, and so 
	$$\int_{\Omega_\alpha} 4^{j_0} \rho^2\, d\mathcal{H}^2 \geq \Mod_\mathrm{C}(\Gamma_{\widetilde{u}_{j_0}})> 4^{j_0}.$$
	Thus $\int_{\Omega_\alpha} \rho^2\, d\mathcal{H}^2 > 1$, which contradicts \eqref{rohma}. We have proved \eqref{ubles}. 

	We sequentially concatenate the curves $\pi_{\Omega_\alpha} \circ \gamma_j \in \hat{\Omega}_\alpha$, $j=1,2,\ldots$, to obtain a curve $\gamma \in \Gamma$ 
	such that the elements of $|\gamma| \cap \mathcal{C}(\Omega_\alpha)$ are $p_\emptyset$, $p_{\widetilde{u}_j}$($j=1,2,\ldots$) and $\{0\}$. Combining \eqref{xy1} and \eqref{ubles}, we arrive at
	\begin{equation} \label{anko}
		\int_{\gamma} \rho \, ds + \sum_{j=1}^\infty \rho(p_{\widetilde{u}_j}) \leq 2. 
	\end{equation} 
	Since $\rho$ is any admissible function for $\Gamma$ that satisfies \eqref{rohma}, the definition of the transboundary modulus and \eqref{anko} show that \eqref{zmo} holds. This gives a contradiction, 
	so we have $\hat{f}(\{0\}) \in \mathcal{C}_N(D)$ for every conformal homeomorphism $f: \Omega_\alpha \to D$ onto a circle domain $D \subset \mathbb{S}^2$. 
\end{proof}


\section{Non-preservation of cofatness under quasiconformal maps } \label{sec:cofatness-not-qc-invariant}

We construct an Ahlfors $2$-regular metric space $X$ which is homeomorphic to $\mathbb{R}^2$, a quasiconformal homeomorphism $F: X\to \mathbb{R}^2$, and a cofat domain $\Omega\subset X$, such that $F(\Omega)\subset \mathbb{R}^2$ is not cofat. The construction can be modified to obtain a metric two-sphere $X$, $\Omega \subset X$, and $F:X \to \mathbb{S}^2$ with the same properties; we omit the details.

\subsection{Metric plane} Let
$$X_{+}:=\bigl\{(u,u^{2},v):u\ge 0,\ v\in\mathbb{R}\bigr\}, \quad 
X_{-}:=\bigl\{(u,0,v):u\ge 0,\ v\in\mathbb{R}\bigr\}, $$
and $X:=X_{+}\cup X_{-}$. We equip $X$ with the restriction of the Euclidean metric of $\mathbb R^{3}$: 
$$d_X(p,q):=|p-q|_{\mathbb R^{3}}.$$
Then $(X,d_X)$ is a metric space and homeomorphic to $\mathbb{R}^2$.

Next, we define $F:X\to \mathbb{R}^2$ by
$$F(u,u^{2},v):=(s(u),v),
\qquad \text{for}\,\,\,(u,u^{2},v)\in X_{+},$$
and 
$$F(u,0,v):=(-u,v),
\qquad  \text{for}\,\,\,(u,0,v)\in X_{-},$$
where
$$
    s(u)
:=
\int_{0}^{u}\sqrt{1+4t^{2}}\,dt \quad \text{for}\quad u\ge0.
$$
Then $F$ is a homeomorphism. Elementary computations show that 
\begin{eqnarray}\label{preservelength}
    \ell_X(\gamma) &=&
\ell_{\mathbb R^{2}}(F\circ\gamma)\quad \text{for every rectifiable curve} \,\, \gamma\,\,\text{in}\,\, X, \\
\label{preservearea}
    \mathcal{H}^{2}_X(E) &=& \mathcal L^{2}(F(E))\quad \text{for every Borel set}\,\, E\subset X.
\end{eqnarray}
In particular, $F$ preserves the conformal modulus of any curve family in $X$, i.e., is $1$-quasiconformal. It is not difficult to see that both $X_+$ and $X_-$ are Ahlfors $2$-regular, hence so is $X$.

\subsection{Cofat domain }
We define
\begin{eqnarray*}
A
&:=&
\bigl\{
(u,u^{2},v)\in X_{+}:u^{2}+v^{2}\leq1
\bigr\}\subset X_+, \\ 
B
&:=&
\bigl\{
(u,0,v)\in X_-:
0\leq u\leq 1/2,\ |v|\leq u^{2}
\bigr\}\subset X_-,  
\end{eqnarray*} 
and 
\begin{equation}\label{defE}
    E:=A\cup B\subset X.
\end{equation}
Note that $A$ and $B$ are closed and connected, and $A\cap B=\{(0,0,0)\}$, so that $E$ is a continuum.

Next, we prove that $E\subset X$ is fat.
Let $$D_{+}
:=
\bigl\{
(u,v)\in\mathbb R^{2}:
u\geq0,\ u^{2}+v^{2}\leq1
\bigr\}.$$
The bi-Lipschitz property of the homeomorphism 
$$\Phi:D_{+}\to A,
\qquad
\Phi(u,v):=(u,u^{2},v) $$ 
guarantees the existence of $C_A>0$ such that
\begin{equation}\label{fatonA}
    \mathcal H^{2}_{X}
\bigl(A\cap B_X(a,r)\bigr)
\geq C_A r^{2}
\end{equation}
for every $a\in A$, $r>0$ for which
$B_X(a,r)$ does not contain $E$. 

It remains to consider points of $B$. We fix $b:=(u_b,0,v_b)\in B$.
If $u_b=0$, then $b=(0,0,0)\in A$, and \eqref{fatonA} applies. We assume that $u_b>0$, and denote $$\widetilde{b}:=(u_b,u_b^{2},v_b).$$
Since $u_b\leq 1/2$ and $|v_b|\leq u_b^{2}$, we have
$$u_b^2+v_b^2\le u_b^2+u_b^4 \le 1, $$
and hence $\widetilde{b}\in A$. We consider three cases separately.  

Suppose first that $r\geq 2u_b^2$. Then $d_X(b,\widetilde{b})=u_b^2\le r^2/2$, and 
$$B_X(\widetilde{b},r/2)\subset B_X(b,r).$$
Therefore, by \eqref{fatonA}, we have
\begin{equation}\label{fatE1}
    \mathcal H^{2}_{X}
\bigl(E\cap B_X(b,r)\bigr)
\geq
\mathcal H^{2}_{X}
\left(
A\cap B_X(\widetilde{b},r/2)\right)\ge C_Ar^2/4.
\end{equation}

Next, we suppose that $0<r\leq u_b^2/8$. 
Note that the two boundary arcs of the outward cusp 
$$\bigl\{
(x,y):
0\leq x\leq 1/2,\ |y|\leq x^{2}
\bigr\}$$ have slopes bounded by
$$\left|\frac{d}{du_b}(\pm u_b^2)\right|
=
2u_b
\leq 1. $$
It follows that there exists $C_B>0$ such
that
\begin{equation}\label{fatB}
    \mathcal H^{2}_{X}
\bigl(E\cap B_X(b,r)\bigr)
\geq \mathcal H^{2}_{X}
\bigl(B\cap B_X(b,r)\bigr)
\geq C_Br^{2}.
\end{equation}

Finally, suppose that ${u_b^{2}}/8<r<2u_b^{2}$.
By \eqref{fatB}, we have
$$\mathcal{H}^{2}_{X}
\bigl(E\cap B_X(b,r)\bigr)
\geq
\mathcal H^{2}_{X}
\left(
B\cap B_X\left(b,{u_b^{2}}/8\right)
\right)
\geq
\frac{C_B}{64}u_b^{4}.$$
Combining with $r<2u_b^2$, we conclude that
\begin{equation}\label{fatE}
    \mathcal H^{2}_{X}
\bigl(E\cap B_X(b,r)\bigr)
\geq
\frac{C_B}{256}r^{2}.
\end{equation}
Combining \eqref{fatonA}-\eqref{fatE} proves that $E$ is $\tau$-fat. 

Let $\Omega:=X\setminus E$. Then $\Omega$ is a domain which is cofat since $E$ is fat. 

\subsection{The image domain is not cofat.}
Finally, we show that $F(E)\subset \mathbb{R}^2$ is not fat. Given $0<t<1/4$, let $b_t:=(-t,0)$ and $r_t:=t/2$.
Since $$F(A)\subset\{(x,y)\in\mathbb R^{2}:x\geq0\},$$
we have $B(b_t,r_t)\cap F(A)=\emptyset.$
This implies that $$F(E)\cap B(b_t,r_t)=F(B)\cap B(b_t,r_t).$$

For every $(-u,v)\in F(B)\cap B(b_t,r_t),$ we have $|u-t|<t/2$, which implies that 
$$t/2<u<3t/2 \quad \text{and}\quad |v|\leq u^{2}\leq 9t^2/4.$$
Therefore, we have 
$$\mathcal{L}^2\bigl(F(E)\cap B(b_t,r_t)\bigr)\le \frac{9}{2}t^3.$$

We conclude that $$\frac{\mathcal{L}^2\bigl(F(E)\cap B(b_t,r_t)\bigr)}{r_t^2}\le 18t\to 0\quad \text{as} \,\, t\to 0.$$
Hence $F(E)\subset \mathbb{R}^2$ is not fat, i.e., $F(\Omega)$ is not cofat.

	\bibliographystyle{plain}
	\bibliography{1234.bib}

\end{document}